\documentclass[12pt]{article}

\usepackage{amssymb}
\usepackage{latexsym}
\usepackage{float}
\usepackage{amssymb, enumerate}

\let\nwline=\newline
\renewcommand{\newline}{\nwline\mbox{}}

\newtheorem{theorem}{Theorem}[section]
\newtheorem{lemma}[theorem]{Lemma}

\newtheorem{corollary}[theorem]{Corollary}
\newtheorem{proposition}[theorem]{Proposition}

\newtheorem{remark}[theorem]{Remark}
\newtheorem{remarks}[theorem]{Remarks}

\newenvironment{proof}{\normalsize {\sc Proof}:}{{\hfill $\Box$ \\}}

\newcommand{\Irr}{{\rm Irr}}
\newcommand{\IBr}{{\rm IBr}}
\newcommand{\Blk}{{\rm Blk}}

\newcommand{\Out}{{\rm Out}}
\newcommand{\GL}{{\rm GL}}

\newcommand{\SU}{{\rm SU}}

\newcommand{\Aut}{{\rm Aut}}
\newcommand{\SL}{{\rm SL}}

\newcommand{\SO}{{\rm SO}}

\newcommand{\Sp}{{\rm Sp}}
\newcommand{\Spin}{{\rm Spin}}

\newcommand{\Syl}{{\rm Syl}}

\def\bF{\mathbb {F}}

\def\oG{\overline {G}}

\def\oH{{\overline {H}}}
\def\oT{{\overline {T}}}
\def\oK{{\overline {K}}}

\def\oC{\overline {C}}

\def\Lie{{\mathcal {L}ie}}

\def\cf{\mathcal{F}}

\def\cE{\mathcal{E}}

\def\BZ{\hbox{\rm Z \kern-.65em Z}}

\def\cO{{\mathcal{O}}}

\begin{document}

\begin{center} {\bf Morita equivalence classes of blocks with extraspecial defect groups $p_+^{1+2}$
\footnote{This research was supported by
the Marsden Fund (of New Zealand), via award numbers UoA 1626 and UoA 2030, and the EPSRC (grant no. EP/T004606/1)}} \\
Jianbei An\footnote{Department of Mathematics, University of Auckland, Auckland, New Zealand. Email: j.an@auckland.ac.nz} and Charles W. Eaton\footnote{Department of Mathematics, University of Manchester, Manchester M13 9PL, United Kingdom. Email: charles.eaton@manchester.ac.uk} \\
\end{center}

\begin{abstract}
We characterise the Morita equivalence classes of blocks with extraspecial defect groups $p_+^{1+2}$ for $p \geq 5$, and so show that Donovan's conjecture and the Alperin-McKay conjecture hold for such $p$-groups. For $p=3$ we reduce Donovan's conjecture for blocks with defect group $3_+^{1+2}$ to bounding the Cartan invariants for such blocks of quasisimple groups. We apply the characterisation to the case $p=5$ as an example, to list the Morita equivalence classes of such blocks.
\end{abstract}



\section{Introduction and main results}
\label{Sec:notation}

Let $p$ be a prime and $(K,\cO,k)$ be a modular system with $k$ an algebraically closed field of characteristic $p$ (we will discuss the choice of $\cO$ further later). Donovan's conjecture, which may be stated over $\cO$ or $k$, predicts that for a given finite $p$-group $P$, there are only finitely many Morita equivalence classes of blocks of finite groups with defect groups isomorphic to $P$. Whilst significant progress on the conjecture has been achieved for the prime two, especially for abelian defect groups, at present for odd primes the only cases for which the conjecture is known are the cyclic $p$-groups  and those $p$-groups admitting no fusion system other than that of the $p$-group itself (see~\cite{el23} for a recent summary, and also~\cite{wiki} where progress on the conjecture is recorded). We apply the classification of finite simple groups to give a description of the Morita equivalence classes of $\cO$-blocks with extraspecial defect groups of order $p^3$ and exponent $p$ when $p \geq 5$, from which it follows that Donovan's conjecture holds for these $p$-groups, and further one may derive a classification of the Morita equivalence classes. As an illustration, we give the explicit classification of all Morita equivalence classes of blocks with defect group $5_+^{1+2}$. A further application is that the Alperin-McKay conjecture holds for blocks with defect group $p_+^{1+2}$ for $p \geq 5$.

In the following note that a Morita equivalence is called basic if it is induced by an endopermutation source bimodule. Our main result is:

\begin{theorem}
\label{main_theorem}
Let $p \geq 5$ be a prime. Let $G$ be a finite group and $B$ a block of $\cO G$ with defect group $D \cong p_+^{1+2}$. Then $B$ is basic Morita equivalent to a block of maximal defect of one of the following:
\begin{enumerate}[(a)]
\item $D \rtimes \hat{E}$, where $\hat{E}$ is a $p'$-central extension of a $p'$-subgroup $E$ of $GL_2(p)$ acting faithfully on $D$;
\item $(C_p \times C_p) \rtimes H$, where $SL_2(p) \leq  H \leq GL_2(p)$;
\item $H$ where $PSL_3(p) \leq H/Z(H) \leq \Aut(PSL_3(p))$ and $Z(H) \leq H'$;
\item $H$ where $PSU_3(p) \leq H/Z(H) \leq \Aut(PSU_3(p))$ and $Z(H) \leq H'$;
\item $2.HS$, $2.HS.2$, $3.McL$, $McL.2$, $2.Ru$, $Co_3$, $Co_2$ or $Th$, where $p=5$;
\item $He$, $He.2$, $3.O'N$, $O'N.2$, $3.Fi_{24}'$ or $Fi_{24}'.2$, where $p=7$;
\item $J_4$, where $p=11$;
\item the Fischer-Griess Monster $M$, where $p=13$.
\end{enumerate}

Further, the Morita equivalence preserves both the fusion of the block and the K\"{u}lshammer-Puig class of the blocks.

\end{theorem}

\begin{corollary}
\label{Donovan}
Donovan's conjecture holds for $\cO$-blocks for $p_+^{1+2}$ for primes $p \geq 5$. That is, for each $p \geq 5$ there are only finitely many Morita equivalence classes amongst $\cO$-blocks of finite groups with defect group $p_+^{1+2}$.
\end{corollary}

Numerical invariants for blocks with these defect groups have been studied in~\cite{he07}, although they are not completely determined, and in~\cite{hks14} it is shown that Olsson's conjecture holds for these blocks, i.e., that the number of irreducible characters of height zero is at most $[D:D']=p^2$. Principal blocks with such defect groups have been studied in~\cite{nu09} using the classification of finite simple groups, where the existence of a generalised perfect isometry is shown. Some of the methods in that paper motivate those used here.   

A consequence of Theorem \ref{main_theorem} is that the Alperin-McKay conjecture holds in this case. Here $k(B)$ denotes the number of irreducible character in $B$. We write $k_h(B)$ for the number of irreducible characters of height $h$, where the height of an irreducible character $\chi$ in a block with defect group $D$ is the non-negative integer $h$ such that $\chi(1)_p=p^h[G:D]_p$. 

\begin{corollary}
\label{alp_mckay}
Let $B$ be a block of a finite group with defect group $D \cong p_+^{1+2}$ for $p \geq 5$. Let $b$ be the Brauer correspondent of $B$ in $N_G(D)$. Then $k_0(B)=k_0(b)$. 
\end{corollary}



Two important components of the proof of Theorem \ref{main_theorem} are the following: the classification of blocks of quasisimple groups with extraspecial defect groups in~\cite{ae11}; an analysis of blocks of quasisimple groups with defect group $C_p \times C_p$ and a non-inner automorphism of order $p$, making use of~\cite{ADL}.

We also obtain detailed information about blocks with defect group $3_+^{1+2}$, found in Proposition \ref{reduced_classification:prop}. However, there are infinite classes of blocks with such defect groups within which Morita equivalence has yet to be decided, so we cannot at present extend Theorem \ref{main_theorem} to $p=3$. However, we can make the following reduction for Donovan's conjecture. For a block $B$ let ${\rm c}(B)$ be the largest entry of its Cartan matrix.

\begin{theorem}
\label{p=3reduce:theorem}
Suppose that there is $c \in \mathbb{N}$ such that ${\rm c}(B) \leq c$ for all quasisimple groups $G$ and all blocks $B$ of $kG$ with defect group $3_+^{1+2}$. Then Donovan's conjecture
holds for $\cO$-blocks with defect group $3_+^{1+2}$.
\end{theorem}

Note that the fact that we need only examine Cartan invariants for these groups rather than Morita equivalence classes stems from Farrell and Kessar's determination of Morita-Frobenius numbers of blocks of quasisimple groups in~\cite{fk19}.

As an example of the application of Theorem \ref{main_theorem}, we give the complete classification of the sixty-two Morita equivalence classes of blocks with defect group $5_+^{1+2}$. Similar classifications could be produced for the remaining primes $p \geq 5$, although the determination of precise Morita equivalence classes for the blocks with normal defect groups would require ad hoc results in each case.

We do not study the groups $p_-^{1+2}$ of exponent $p^2$ in this article. Whilst many of the reductions for blocks with these defect groups are the same as for the exponent $p$ case, a major difference is that it would require the examination of blocks of quasisimple groups with cyclic defect groups of order $p^2$, which would involve a very different analysis. 

The structure of the paper is as follows. In Section \ref{background} we give some basic notation and prove some elementary general results that we will need later. We introduce notation necessary for the analysis of classical groups in Section \ref{Sec:notationcgps}. The analysis of blocks of quasisimple groups with defect group $C_p \times C_p$ and a non-inner automorphism of order $p$ is contained in Sections \ref{cgps} and \ref{egps}. We prove Theorem \ref{main_theorem} and Theorem \ref{p=3reduce:theorem} in Section \ref{reductions:section}. We analyse the case $p=5$ in full detail in Section \ref{p=5}. In Section \ref{counting_conjectures} we prove Corollary \ref{alp_mckay}.


\section{Notation and background results}
\label{background}

Let $p$ be a prime and let $\cO$ be a complete discrete valuation ring with algebraically closed residue field $k$ of characteristic $p$, so that $k=\mathcal{O}/J(\mathcal{O})$. Let $\mathcal{K}$ be the field of fractions for $\cO$. The results of this paper hold in this generality, however for a canonical statement of Donovan's conjecture one may take $k$ to be the algebraic closure of the field $\mathbb{F}_p$ of $p$ elements and $\cO$ to be the ring of Witt vectors for $k$. 

For $G$ a finite group, write $\Blk(G)$ for the set of blocks of $\mathcal{O}G$. Denote by
$\Irr(B)$ the set of irreducible $\mathcal{K}$-characters belonging to a block $B \in \Blk(G)$. We will often use $D=D(B)$ to denote a defect group of $B$. Denote by $\IBr(B)$ for the set of irreducible Brauer characters in $B$. Write $k(B)=|\Irr(B)|$ and $l(B)=|\IBr(B)|$. Write $B_0(G)$ for the principal block of $\cO G$.

A $B$-subgroup is a
pair $(Q,b_Q)$, where $Q$ is a $p$-subgroup of $G$ and $b_Q$ is a
block of $C_G(Q)$ with Brauer correspondent $(b_Q)^G=B$. The
$B$-subgroups with $|Q|$ maximized are called the Sylow
$B$-subgroups, and they are the $B$-subgroups for which $Q$ is a defect group 
for $B$. We denote by $N_G(D,b_D)$ the stabilizer in $N_G(D)$ of $(D,b_D)$ under conjugation.

The \emph{inertial quotient} of $B$ is $E=N_G(D,b_D)/DC_G(D)$, together with the action of $E$ on $D$. The inertial quotient is determined by the fusion system $\mathcal{F}=\mathcal{F}_{(D,b_D)}(G,B)$ for $B$, sometimes called the Frobenius category. We refer to~\cite[Section 8.5]{lin2} for background on this. Following the presentation in~\cite[Section 8.14]{lin2}, a K\"ulshammer-Puig class is an element of $H^2(\Aut_\mathcal{F}(Q))$, which is isomorphic to $H^2(E,k^\times)$ (see~\cite[Remark 8.14.3]{lin2}). 

The Morita equivalences occurring paper are all basic Morita equivalences or source algebra equivalences (a basic Morita equivalence is one induced by an endopermutation source bimodule, for which see~\cite[Section 9.10]{lin2}).

Now suppose $D \lhd G$ and let $E$ be the inertial quotient of $B$. By~\cite[Theorem 6.14.1]{lin2}, $B$ is basic Morita equivalent to a twisted group algebra $\mathcal{O}_\alpha(D \rtimes E)$, where $\alpha \in H^2(E,k^\times)$. In turn, this is basic Morita equivalent to a block of $\cO (D \rtimes \hat{E})$, where $\hat{E}$ is a central extension of $E$ by $Z \leq Z(\hat{E}) \cap \hat{E}'$ with $E$ acting faithfully and $Z$ acting trivially on $D$. Blocks of $\cO (D \rtimes \hat{E})$ correspond to irreducible characters of $Z$. We remark that the Morita equivalence classes are independent of choice of isoclinism class of central extension. Whilst choice of $E$ and $\alpha$ determines the Morita equivalence class, it is not always the case that different choices give distinct Morita equivalence classes. It may (and does) happen that distinct choices of $\alpha$ give the same Morita equivalence class, as we will see in Section \ref{p=5}.

The outer automorphism group of $p_+^{1+2}$ is isomorphic to $GL_2(p)$, and the possible inertial quotients correspond to the conjugacy classes of subgroups of $GL_2(p)$ of order prime to $p$.

\medskip

We now gather together some elementary results that will be needed later on.

\medskip

A main task in this paper is to show that if a quasisimple group of Lie type in characteristic different to $p$ has a block with defect group $C_p \times C_p$, then there cannot be a field automorphism of order $p$. The reason for this is roughly that if the field of definition of our group has exponent divisible by $p$, then such small but nontrivial defect groups as these cannot be present. The next lemma is important in showing this. Throughout, we denote by $m_p$ the $p$-part of an integer $m$.
 
Suppose $p$ is odd. Let $r$ be a prime distinct from $p$ and $q=r^f$ for some integer $f\geq 1$. Let $\epsilon\in 
\{+1, -1\}$ and let $e$ be the multiplicative order of $\epsilon q$ modulo $p$, so that $p\mid (q^e-\epsilon)$. Write $(q^e-\epsilon)_p=p^a$, where $a\geq 1$ is an integer.

\begin{lemma}
\label{Lem:a=1}
With the above notation, if $a=1$, then $p\nmid f$.
\end{lemma}

\begin{proof}
Suppose first that $\epsilon = 1$. Let $e'$ be the multiplicative order of $r$ modulo $p$. Calculating in $C_p$, we have $q^{e'}=r^{fe'}=1$, so $e\mid e'$ and $e'=e e_1$ for some $e_1\geq 1$. But $r^{e f}=1$, so $e'\mid (ef)$ and write $ef =e' f_1$. Thus
\[
ef=e'f_1 =e e_1 f_1
\]
and $f=e_1 f_1$. Let $r_1=r^{e_1}$, so that $|r_1|=e$. If $p^b=(r_1^e-1)_p$ with $b\geq 1$, then $r_1^e= p^b m+1$ for some integer $m\geq 1$ with $\gcd (m, p)=1$. 
If $p\mid f$, then $p\mid f_1$ and $f_1=pf_2$ as $e_1\leq e'\leq p-1$.  
Thus  $q^e=(r_1^e)^{f_1}=((p^b m+1)^{f_2})^p$. In particular, $p^{b+1}\mid (q^e-1)$ and
hence $a\geq b+1\geq 2$, a contradiction. 

Suppose $\epsilon =-1$. Then $p\mid |U(e,q)|$ with $e$ minimal.
If $x\in U(e,q)$ has order $p$, then $Z(U(e,q))=C_{q+1}\leq C_{U(e,q)}(x)=U(1, q^e)$ and so $e$ is odd.  In particular, $q^e\equiv -1 \pmod {p}$ and $p\mid (q^e+1)$.
Since $p$ is odd, we have $p\nmid (q^e-1)$ and $(q^{2e}-1)_p=(q^e+1)_p$. 
So if $a=1$, then $(q^{2e}-1)_p=p$ and as shown above, $p\nmid 2f$ and hence $p\nmid f$. 
\end{proof}

\begin{lemma}
\label{subs_PSL2:lemma}
Let $H \leq GL_2(p)$ such that $H$ contains a Sylow $p$-subgroup $P$ of $GL_2(p)$. Then $P \lhd H$ or $SL_2(p) \leq H \leq GL_2(p)$.
\end{lemma}

\begin{proof}
This is well known and follows easily from~\cite[Chap. II, 8.27]{hup}.
\end{proof}

\begin{lemma}
\label{GL_2(p)splitting:lemma}
Let $G$ be a finite group with $C_G(O_p(G)) \leq O_p(G) \cong C_p \times C_p$ and $|G|_p=p^3$, where $p$ is an odd prime. Then $G = O_p(G) \rtimes H$ for some $SL_2(p) \leq H \leq GL_2(p)$.
\end{lemma}

\begin{proof}
We may embed $G/O_p(G)$ into $\Aut(O_p(G)) \cong GL_2(p)$. Let $P \in \Syl_p(G)$. Since $O_p(G) \neq P$, we must have that $G/O_p(G)$ is isomorphic to a subgroup of $GL_2(p)$ of order divisible by $p$ that does not normalise a Sylow $p$-subgroup. Hence by Lemma \ref{subs_PSL2:lemma} $SL_2(p) \leq G/O_p(G)$ and it suffices to show that $G$ must be a split extension. This follows from~\cite[Theorem 1]{cu76} by observing that $Z(SL_2(p))$ fixes only the zero element of the natural module.
\end{proof}


\section{Notation for classical groups}
\label{Sec:notationcgps}

Let $V$ be a linear, unitary, non-degenerate orthogonal or
symplectic space over the field $\bF_q$, where $q=r^f$ with $\gcd(r,p)=1$. We will follow the notation of \cite{An94},
\cite{B86}, \cite{FS82} and~\cite{FS89}.

If $V$ is orthogonal, then there is a choice of equivalence classes
of quadratic forms. Write $\eta(V)$ for the type of $V$ as defined
in~\cite{FS89}, so $\eta(V)=+$ or $-$. In addition, let $\eta(V)=+$ or $-$ according as $V$ is linear or unitary. If $V$ is non-degenerate orthogonal or symplectic, then denote by $I(V)$ the group of isometries on $V$ and let $I_0(V)=I(V)\cap \SL(V)$.

If $V$ is symplectic, then $I(V)=I_0(V)={\Sp}_{2n}(q)$.

If $V$ is a $(2n+1)$-dimensional orthogonal space, then
$I(V)=\langle -1_V\rangle\times  I_0(V)$ with
$I_0(V)=\SO_{2n+1}(q)$.

If $V$ is a $2n$-dimensional orthogonal space, then
$I(V)=O^\eta(V)=O^\eta_{2n}(q)$ and $I_0(V)=\SO^\eta_{2n}(q)$.

If $V$ is a $2n$-dimensional non-degenerate orthogonal or symplectic
space, then denote by $J_0(V)$ the conformal isometries of $V$ with
square determinant. If $V$ is orthogonal of dimensional at least
two, then write $D_0(V)$ for the special Clifford group of $V$ (cf.
\cite{FS89}).

Denote by $\GL^+(V)$ the general linear group $\GL(V)$ and $\GL^-(V)$ the
unitary group $U(V)$.

Let $G=\GL^\eta(V)$ or $I(V)$. Write $\cf_q=\cf_q(G)$ for the set of
(monic) polynomials serving as elementary divisors for semisimple
elements of $G$ (cf. \cite[p.6]{An94}). If $G=\GL^{-}(V)$, then
partition $\cf_q$ as in~\cite{FS82} by writing:
$$\begin{array}{l}
\cf_1=\{ \Delta \in \cf_q :  \Delta \ {\rm irreducible}, \
\Delta \neq X, \ \Delta={\tilde{\Delta}} \} , \\
\cf_2=\{ \Delta {\tilde{\Delta}} \in \cf_q : \Delta \ {\rm irreducible}, \
\Delta \neq X, \ \Delta \neq {\tilde{\Delta}} \} ,
\end{array}$$
where ${\tilde{\Delta}}$ is the monic irreducible polynomial whose roots are
the $q$-th power of those of $\Delta$. If $G=I(V)$, then partition $\cf_q$ as
in~\cite{FS89} by writing: $$\begin{array}{l}
\cf_0=\{X \pm 1 \} , \\
\cf_1=\{ \Delta \in \cf_q \setminus
 \cf_0: \Delta \ {\rm irreducible}, \
\Gamma \neq X, \ \Delta=\Delta^* \} , \\
\cf_2=\{ \Delta\Delta^* \in \cf_q \setminus
 \cf_0: \Delta \ {\rm irreducible}, \
\Delta \neq X, \ \Delta \neq \Delta^* \} ,
\end{array}$$
where $\Delta^*$ is the monic irreducible polynomial whose roots are the
inverses of those of $\Delta$. Let $d_\Gamma$ be the degree of
$\Gamma\in\cf_q$, and define the reduced degree $\delta_\Gamma$ as
in~\cite{An94} and \cite{FS89}, so that $\delta_\Gamma = d_\Gamma$ if
$G=\GL^\eta(V)$ or $\Gamma \in \cf_0$, and $\delta_\Gamma =
\frac{1}{2}d_\Gamma$ if $G=I(V)$ and $\Gamma\in \cf_1\cup\cf_2$.


For each $\Gamma \in \cf_q$, define
$\epsilon_\Gamma=1$ when $G=\GL(V)$, otherwise define
$$\epsilon_\Gamma = \left\{ \begin{array}{cll}
\epsilon & if & \Gamma \in \cf_0 , \\
-1 & if & \Gamma \in \cf_1 , \\
1 & if & \Gamma \in \cf_2 . \end{array} \right.$$

Let $e_\Gamma$ be the multiplicative order of $\epsilon_\Gamma
q^{\delta_\Gamma}$ modulo $p$. Thus we may write
$e_\Gamma\delta_\Gamma=ep^{\alpha_\Gamma}\delta_\Gamma'$ for some
$\alpha_\Gamma$ and $\delta_\Gamma'$ with  $p\nmid \delta_\Gamma'$.

Let $\cf_q^{p'}=\cf_q^{p'}(G)$ be the subset of $\cf_q$ serving as 
elementary divisors for all semisimple $p'$-elements of $G$.

Given a semisimple element $s \in G$, there is a unique orthogonal
decomposition $V=\sum_{\Gamma \in \cf_q} V_\Gamma(s)$, with $s =
\prod_{\Gamma \in \cf_q} s_\Gamma$, where the $V_\Gamma(s)$ are
nondegenerate subspaces of $V$ and $s_\Gamma \in U(V_\Gamma(s))$,
$\GL(V_\Gamma(s))$ or $I(V_\Gamma(s))$ (depending on $G$) has minimal
polynomial $\Gamma$. (For simplicity, we write a direct sum as orthogonal if $G=\GL(V)$). The decomposition is called the primary decomposition of
$s$. If $s=s_\Gamma$, then $s$ is called {\it primitive}.
Write $m_\Gamma(s)$ for the multiplicity of $\Gamma$ in
$s_\Gamma$. We have $C_G(s)=\prod_{\Gamma \in \cf_q} C_\Gamma(s)$,
where $C_\Gamma(s)=I(V_\Gamma(s))$ or
$\GL^{\epsilon_\Gamma}(m_\Gamma(s),q^{\delta_\Gamma})$ as appropriate.


\section{Classical groups}
\label{cgps}

In this and the following section we analyse blocks of quasisimple groups in non-defining characteristic with defect group $C_p$ or $C_p \times C_p$. We show that if such a block exists, then there cannot be a field automorphism of order $p$. This will help to rule out the existence of blocks of automorphism groups of such groups with extraspecial defect groups.

\bigskip
Suppose $\eta\in \{+1, -1\}$ and $G=\GL_n^\eta(q)=\GL^\eta(V)$, and let $B_G$ be a
$p$-block of $G$ with a defect group $D_G$ and label $(s,\kappa)$. Then
$$
V =V_0\perp V_+,\quad D_G =D_0\times D_+,\quad  s=s_0\times s_+,
\eqno(\ref{cgps}.1)
$$
where $V_0=C_V(D_G)$, $V_+=[D_G, V]$, $s_0\in G_0= \GL^\eta(V_0)$ and
$s_+\in G_+:=\GL^\epsilon(V_+)$, $D_0=\langle 1_{V_0}\rangle$ and $D_+\leq G_+$. We also denote $\GL^\eta(V)$ by $G(V)$ and
$\SL^\eta(V)$ by $S(V)$.


\begin{proposition}
\label{Prop:slu}
Let $K_u=\SL^\eta(V)=\SL^\eta_n(q)$, $G=\GL^\eta(V)$, $K=K_u/Z$
for some $Z\leq Z(K_u)$, $B\in\Blk(K)$ with Sylow $B$-subgroup
$(D, b_K)$. Take $B_u\in\Blk(K_u)$ dominating $B$, 
$B_G\in\Blk(G)$ covering $B_u$, $D_u=D(B_u)$ such that $D_uZ/Z=D$ and
$D_G=D(B_G)$ such that $D_u=D_G\cap  K_u$. Suppose that $D \cong C_p$ or $C_p\times C_p$. Then either any block of $G/Z$ covering $B$ has a cyclic defect group, or
$a=1$ and hence $K$ has no field automorphism of order $p$.
\end{proposition}

\begin{proof}
Let $(s, \kappa)$ be the label of $B_G$, and $V, D_G, s$ have the
corresponding decomposition  (\ref{cgps}.1). Let $s_+=\prod_{\Gamma}s_{\Gamma}$
be a primary decomposition, so that $V_+=\bigoplus_{\Gamma} V_\Gamma$ with
$V_\Gamma$ the underlying space of $s_{\Gamma}$. Thus
$$
C_{G_+}(s_+)=\prod_{\Gamma} C_\Gamma, \quad C_\Gamma=
\GL^{\epsilon_\Gamma}(m_\Gamma, q^{\delta_\Gamma})\eqno(\ref{cgps}.2)
$$
with $m_\Gamma=m_\Gamma(s_+)$. We may suppose $D_+\in\Syl_p(C_{G_+}(s_+))$, that is, a Sylow subgroup of $C_{G_+}(s_+)$, so that
$$
D_+=\prod_{\Gamma} D_\Gamma,\quad D_\Gamma\in\Syl_p(C_\Gamma).
\eqno(\ref{cgps}.3)
$$
So $D_G$ is a direct product of cyclic groups and wreath product $p$-groups.
Since $C_{V_\Gamma}(D_\Gamma)=\{0\}$, it follows that $m_\Gamma=e_\Gamma w_\Gamma$
for some integer $w_\Gamma$, and 
\[
D_\Gamma \in \Syl_p(\GL^{\epsilon_\Gamma}(w_\Gamma, q^{e_\Gamma \delta_\Gamma})).
\]

Suppose $e\neq 1$ or $\eta\neq \epsilon$, so that $Z(G)=C_{q-\eta}$ is a $p'$-group. Thus $D\cong D_u=D_G$.
In particular, $D_G$ is abelian and each $w_\Gamma<p$.
Thus $D_\Gamma \cong (C_{p^{a+\alpha_\Gamma}})^{w_\Gamma}$.
Note that $w_\Gamma\geq 1$ for some $\Gamma$.  Since $D_\Gamma\leq C_p\times C_p$, we have $a=1$ and $\alpha_\Gamma=0$. 

Suppose $e=1$ and $\eta=\epsilon$, so that $(q-\epsilon)_p=p^a$ and each $e_\Gamma=1$. If $w_\Gamma\geq 2$, then define $R_\Gamma=\langle {\mathrm {diag}}(x, x^{-1}, I_{w_\Gamma-2}) \mid x\in C_{p^{a+\alpha_\Gamma}}\rangle$, a subgroup of $D_\Gamma$, with $R_\Gamma\leq S(V_\Gamma)$ and 
$R_\Gamma \cap Z(S(V_\Gamma))=1$. So $R_\Gamma\leq D_u$ and $R_\Gamma\cong R_\Gamma Z/Z\leq D$. 
Since $R_\Gamma\cong C_{p^{a+\alpha_\Gamma}}$, we have that $a=1$ and $\alpha_\Gamma=0$.

Hence it suffices to consider $\Gamma$ with $w_\Gamma=1$,  $D_\Gamma\cong C_{p^{a+\alpha_\Gamma}}$ and 
$\dim V_\Gamma =d_\Gamma$. In particular, $D_G$ is abelian.

If $D_G$ is cyclic, then so is $D_GZ/Z$. Thus a defect group of any block of $G/Z$ covering $B$ is cyclic. Suppose $D_G$ is non-cyclic.

Note that $C_\Gamma\cong \GL^{\epsilon_\Gamma}(1, q^{\delta_\Gamma})$ is a torus and
\[
C_\Gamma\cap S(V_\Gamma)\cong C_{q^{\delta_\Gamma-1}+\epsilon_\Gamma q^{\delta_\Gamma-2}+\ldots +1}.
\]

If $\alpha_\Gamma=0$, then $p\nmid \delta_\Gamma$,  $D_\Gamma=Z(G(V_\Gamma))_p=C_{p^a}$ and so $D_\Gamma\cap S(V_\Gamma)=1$.

Since $D_G$ is noncyclic, we have $m_\Delta=w_\Delta =1$ for some $\Delta\neq \Gamma$.
Now $\det : C_\Gamma\times  C_\Delta \to \bF_q^\times$ is surjective, so 
$\det : D_\Gamma\times D_\Delta \to  C_{p^a}=O_p(\bF_q^\times)$ is onto. In particular,
$(D_\Gamma\times D_\Delta)\cap S(V_\Gamma\perp V_\Delta)$ contains a cyclic group
$P=\{{\mathrm {diag}}(x, y) \mid x\in D_\Gamma, y=y(x)\in  D_\Delta\}$, where 
$\det(y(x))=\det (x)^{-1}$. Since $\det(x)\neq 1$, we have that $PZ/Z=P\leq D$.
In particular, $a=1$.  

Suppose each $\alpha_\Gamma\geq 1$. Then $D_\Gamma/Z(G(V_\Gamma))_p\cong C_{p^{\alpha_\Gamma}}$. A proof similar to above shows that $\alpha_\Gamma\leq 1$. Thus $\alpha_\Gamma=1$ and $(\delta_\Gamma)_p=p$. So $D_\Gamma\cap  S(V_\Gamma)\cong C_p$.
Since $|(D_\Gamma\times D_\Delta)\cap S(V_\Gamma\perp V_\Delta)|\geq p^{a+\alpha_\Gamma+\alpha_\Delta}$ and $D=D_u  Z/Z=C_p\times C_p$, we have $m_{\Gamma'}(s_+)=0$ for
$\Gamma'\not\in \{\Gamma, \Delta\}$, so 
\[
D_+=D_\Gamma\times D_\Delta\cong C_{p^{a+\alpha_\Gamma}}\times C_{p^{a+\alpha_\Delta}}=
C_{p^{a+1}}\times C_{p^{a+1}}.
\]

Let $Z_\Gamma=Z(G(V_\Gamma))_p\cong C_{p^a}$. Suppose $a\geq 2$. A similar proof to above shows that $D_+\cap S(V_+)$ contains a subgroup $P=\{{\mathrm {diag}}(x, y) \mid x\in Z_\Gamma, y=y(x)\in  D_\Delta\}$, where $\det(y(x))=\det (x)^{-1}$. Since $(d_\Gamma)_p=p$, we have $\det(x)\neq 1$ for generators $x\in Z_\Gamma$, so $P\cap Z(S(V_+))=1$. Hence
$PZ/Z\cong P \leq D$, which is impossible. So $a=1$.
\end{proof}


\bigskip
Let $V$ be a non-degenerate orthogonal or symplectic space,
$G=I_0(V)$ and let $G^*$ be the dual group of $G$. Note that
\[
\Sp_{2n}(q)^*=\SO_{2n+1}(q),\quad  \SO_{2n+1}(q)^*=\Sp_{2n}(q),\quad
\SO_{2n}^\eta(q)^*=\SO_{2n}^\eta(q).
\]
If $B$ is a $p$-block of $I_0(V)$, then there exists a semisimple $p'$-element
$s\in I_0(V)^*$ such that
\[
B\subseteq\cE_p(I_0(V),(s)).
\]
Let $(D, b_D)$ be a Sylow $B$-subgroup of $I_0(V)$. Then $V$ and
$D$ have corresponding decompositions
$$
V = V_0\perp V_+,\quad D=D_0\times D_+.\eqno(\ref{cgps}.8)
$$
If $p$ is odd, then $V_0=C_V(D)$, $V_+=[ V,D ]$,
$D_0=\{1_{V_0}\}$ and $D_+\leq I_0(V_+)$.  Let $G_0=I_0(V_0)$, $G_+=I_0(V_+)$,
$C_+=C_{I_0(V_+)}(D_+)$ and let $V^*$ be the underlying space of
$I_0(V)^*$.

Let $z\in D$ be a primitive element.  Then
$$
z=z_0\times z_+,\quad L=C_G(z)=L_0\times L_+,\quad L_0=G_0, \quad
L_+=\GL^\epsilon(m,q^e),
\eqno(\ref{cgps}.9)
$$
where $z_0=1_{V_0}$, $z_+\leq D_+$ and $\dim V_+=2em$.
Then $L$ is a Levi subgroup of $G$ and we may suppose
$s\in L^*\leq G^*$. In particular,
$$
V^*=U_0\perp U_+ \quad \hbox{   and    }\quad s=s_0\times s_+,
\eqno(\ref{cgps}.10)
$$
where $U_0=V_0^*$, $s_0\in L_0^*=I_0(U_0)$, $U_+$ is the
underlying space of $L_+^*$ and  $s_+\in L_+^*\leq I_0(U_+)$.


\begin{proposition}
\label{Prop:spomega}
Let $p$ be an odd prime

(a) Suppose $K_c\in\{\Sp_{2n}(q), \Omega^\eta_{2n}(q), \Omega_{2n+1}(q)\}$ and $K=K_c/Z_c$ for some $Z_c\leq Z(K_c)$. If $B\in\Blk(K)$ with $D_K=D(B)\cong C_p$ or $C_p\times C_p$, then $a=1$. In particular, $K$ has no field automorphism of order $p$.

(b) let $V$ be orthogonal, $K_u={\rm Spin}^\eta(V)$ and $K=K_u/Z_u$ for some $Z_u\leq Z(K_u)$, where $\eta=\circ$ when $\dim V$ is odd. Suppose $B_K\in \Blk(K)$ with $D_K=D(B_K)\cong C_p$ or $C_p\times C_p$. Then $a=1$ and $K$ has no field automorphism of order $p$. 

(c) Let $V$ be orthogonal such that $K_u=\Spin_8^+(q)=\Spin(V)$, $Z_u=1$ or $Z(K_u)$ and 
$K=K_u/Z_u$. Let $\Phi_K$ and $\Gamma_K$ be the field and graph automorphism groups of $K$, respectively, so that $\Phi_K\cong C_f$ and $\Gamma_K\cong S_3$. Let
\[
K_u\leq A\leq K_u.\Out(K),
\]
and $A_u=A/Z_u$, where $\Out(K)= \Phi_K\times {\rm Outdiag}(K_u).\Gamma_K$ with
${\rm Outdiag}(K).\Gamma_K\cong S_3$ or $S_4$ according to $r=2$ or odd.  

Suppose $p=3$ and $D\leq A_u$ is a defect group such that $D\cong 3_+^{1+2}$ and $D_K=D\cap K\cong C_3$ or $C_3\times C_3$.

(i) Then $C_K(D)=\SL_3^\epsilon(q)$ and 
\[
N_{A_u}(D)=(3_+^{1+2}\circ_3 \SL_3^\epsilon(q)).(C_d\times S_3\times \Phi)
\]
where $C_d=1$ or $C_2$ and $\Phi$ is a $3'$-subgroup of $\Phi_K$. Moreover, the outer side group $S_3$ in $N_{A_u}(D)$ is generated by $x$ and $y$ with $|x|=3$ and $|y|=2$ such that $x$ and 
$y$ induce outer-diagonal and graph automorphisms on $\SL_3^\epsilon(q)$, respectively.

(ii) The subgroup $D\cong 3_+^{1+2}$ given above is not a defect group of $A_u$.  

\end{proposition}

\begin{proof}
(a) Let $V$ the underlying space of $K_c$ and $B_c\in\Blk(K_c)$ dominating $B$. Then $D_c\cong D_cZ_c/Z_c\cong D_K\cong C_p$ or $C_p\times C_p$ for some
$D_c=D(B_c)$. Let $G=I_0(V)$ and $B_G\in\Blk(G)$ covering $B_c$. Then
$D_c=D\cap K_c$ for some $D=D(B_G)$. Since $I_0(V)/K_c$ is a $p'$-group, we have
$D_c=D$. 

Follow the notation above. We have $D=D_0\times D_+\cong D_+$. Identify $L_+$ with $L_+^*$. Then $D_+\in \Syl_p(C_{L_+}(s_+))$. In particular, $C_{p^a}\cong Z(L_+)_p\leq D_+$ and hence $a=1$.

\medskip
(b) Let $Z\leq Z(K_u)$ such that $K_c=\Omega^\eta(V)=K_u/Z$ and let $B_u\in \Blk(K_u)$
dominate $B$ and $B_c\in \Blk(K_c)$ dominated by $B_u$, Then 
$D_K=D_u Z_u/Z_u$ for some $D_u=D(B_u)$ and $D_uZ/Z=D(B_c)$.
Since $Z$ and $Z_u$ are $p'$-group, we have $D(B_c)\cong D_u\cong D_K\cong C_p$ or $C_p\times C_p$.
By Proposition \ref{Prop:spomega} (a), $a=1$. 

\medskip
(c) (i) Let $B$ be a block of $A_u$ with defect group $D$ and let $B_K$ be a block of $K$ covered by $B$. We may suppose $D(B_K)=D_K$, wich is isomorphic to $C_3$ or $C_3\times C_3$. By 
parts (a) and (b) above, $K$ has no field automorphism of order $p$. So we may suppose
$A\leq K_u.S_3$ or $A\leq K_u.S_4$ according to $r=2$ or odd. 

Since a Sylow $3$-subgroup of $\Out(K)$ has order $3$, we have that $D_K\cong C_3\times C_3$ and 
$Z(D)\leq D_K$. Thus $D$ contains an order $3$ graph automorphism $\gamma_K$ of $K$.
Note that $\gamma_K$ acts on $Z(K_u)\setminus \{1\}$ transitively.
So we may assume that $Z_u=Z(K_u)$ and $K={\rm P}\Omega_8^+(q)$. By \cite[Table 4.7.3A]{gls98}, 
\[
C_K(Z(D))=(C_{q-\epsilon}\times C_{q-\epsilon})\circ_3 (\SL_3^\epsilon(q).x),
\]
where $x$ induces the order $3$ outer diagonal automorphism on $\SL_3^\epsilon(q)$. 
In addition, $N_K(Z(D))=C_K(Z(D)).y$ and 
\[
\Out_{{\rm Outdiag}(K).S_3}(Z(D))\cong \langle \gamma_K,\rho_K\rangle\cong S_3,
\]
where $y$ induces a graph automorphism on $\SL_3^\epsilon(q)$ and inveres elements of
$(C_{q-\epsilon}\times C_{q-\epsilon})$, $\gamma_K$ induces $(\omega{:}1)$ on
$C_K(Z(D))$ (there is a typo in \cite[Table 4.7.3A]{gls98}, see \cite[Table 4.7.2]{gls98}) and 
$\langle \gamma_K,\rho_K\rangle=\Gamma_K$. So $D_K=O_3(C_{q-\epsilon}\times C_{q-\epsilon})$ and 
$D=\langle D_K,\gamma_K\rangle\cong 3_+^{1+2}$. Since $N_{A_u}(D))\leq N_{A_u}(Z(D))$, it follows that $C_{A_u}(D)=\SL_3^\epsilon(q)$ and 
\[
N_{A_u}(D)=(3_+^{1+2}\circ_3 \SL_3^\epsilon(q)).(\langle x, y\rangle\times 
\langle\rho\rangle),
\]
where $\rho=1$ or $\rho_K$ according to 
$\rho_K\not\in A_u$ or  $\rho_K\in A_u$. Thus (c) (i) holds.

(ii) Follow the notation of (i). Let $(D, b_D)$ be a Sylow $B$-subgroup, so that 
$b_D$ is a block of $C_{A_u}(D)=\SL_3^\epsilon(q)$. If $\theta$ is the canonocal character of $b_D$, then $Z(C_{A_u}(D))$ is in the kernel of $\theta$ and $\theta$ has $3$-defect $0$ as character of $C_{A_u}(D)/Z(C_{A_u}(D))$. By \cite[Lemma 6.12]{AnHissLubeck21}, $\theta$ extends to ${\rm PGL}_3^\epsilon(q)$. Since $x\in N_{A_u}(D)$ induces an order $3$ diagonal outer 
automorphism on $\SL_3^\epsilon(q)$, it follows that $x\in N_{A_u}(D, b_D)$. In particular,
$N_{A_u}(D, b_D)/(C_{A_u}(D)D)$ is not a $3'$-group and hence $(D, b_D)$ is not a Sylow $B$-subgroup, which is impossible. This proves (c) (ii).
\end{proof}


\section{Exceptional groups}
\label{egps}

We repeat the analysis in the previous section for the exceptional groups. In this case we cannot rule out the existence of a field automorphism of order $p$, but the cases in which one may occur are limited, and must centralise the defect groups.

We will follow the notation of \cite{gls98}. Let $K$ be a quasisimple group of exceptional type and $B\in \Blk(K)$ with defect group $D\cong C_p\times C_p$. Let $({\overline {K}}, \sigma)$ be a $\sigma$-setup for $K$ in the sense of \cite[Definition 2.2.1]{gls98}, and let $\oT\leq \oK $ be a maximal torus. A subgroup $A\leq K$ is {\it toral} if
$A$ is a subgroup of a maximal torus of $\oK$. Two elements $w, y\in \oK$ are {\it $\sigma$-conjugate} if $w=xy\sigma(x)^{-1}$ for some $x\in \oK$.

We need the following proposition

\begin{proposition}{\rm (\cite[Propositions  4.1, 4.3 and 4.4 and Lemma 5.2]{ADL})}\label{prop:eleoK} 

Let $A \leq K$ be a subgroup. Then the following hold.

\begin{enumerate}

\item[$(a)$]
There is a bijection between $(\oK)^\sigma$-classes of subgroups of $(\oK)^{\sigma}$ which are 
$\oK$-conjugate to $A$, and $\sigma$-classes in $N_{\oK}(A)/C_{\oK}(A)^{\circ}$ contained in $C_{\oK}(A)/C_{\oK}(A)^{\circ}${:} the $\sigma$-class of $w \in C_{\oK}(A)/C_\oK(A)^\circ$ corresponds to the $(\oK)^\sigma$-class of subgroups with representative $A_{w} = {^g}A$, where $g \in \oK$ is chosen with $g^{-1}\sigma(g)C_{\oK}(A)^{\circ} = w$.

\item [$(b)$] Let $A_w\leq (\oK)^\sigma$ be the $\oK$-conjugate of
$A$ as in $(a)$. If $\dot{w}\in C_\oK(A)$ is any lift of $w$, then
\[(C_{\oK}(A_{w})^\circ)^{\sigma} \cong (C_{\oK}(A)^{\circ})^{\dot{w}\sigma} \]
where $(C_{\oK}(A)^{\circ})^{\dot{w}\sigma}$ is the set of fixed points of $\dot{w}\sigma$ in 
$C_{\oK}(A)^{\circ}$. 

\item [$(c)$] If $A_w$ is as in $(a)$ and  $w\sigma$ is identified with the map $x \mapsto w\sigma(x)w^{-1}$, then
\[
(N_{\oK}(A_{w}))^\sigma/(C_{\oK}(A_{w})^{\circ})^\sigma \cong (N_{\oK}(A)/C_{\oK}(A)^{\circ})^{w\sigma}
\]
and $C_{\oG}(A_{w})^\sigma/(C_{\oK}(A_{w})^{\circ})^\sigma \cong (C_{\oK}(A)/C_{\oK}(A)^{\circ})^{w\sigma}$

\item [$(d)$]
If, moreover that $A\leq \oT\cap K$ is  toral, then 
$N_\oK(A)/C_\oK(A)^\circ = N_{\oK^\sigma}(A)/(C_\oK(A)^\circ)^\sigma$, $C_\oK(A)/C_\oK(A)^\circ =
C_{\oK^\sigma}(A)/(C_\oK(A)^\circ)^\sigma$ and $\sigma$ acts trivially on $N_\oK(A)/C_\oK(A)^\circ$.

\end{enumerate}
\end{proposition}

\noindent
Thus if $\sigma$ acts as identity on $H=N_{\oK}(A)/C_{\oK}(A)^{\circ}$, then $H^{w\sigma}=C_H(w)$ and the
$\sigma$-classes are the $H$-conjugacy classes. So toral elementary abelian $p$-subgroups of a finite exceptional group of Lie type can be determined by \cite{ADL} and \cite{ADLmagma} (cf. \cite[Section 5]{ADL}).
If $D=\langle  x, y\rangle\cong C_p\times C_p$ and $K=K_u$, then $C_\oK(x)$ is connected and so $y$ lies in a maximal 
torus of $C_\oK(x)$. In particular, $D$ is toral. 

\begin{proposition}
\label{Prop:excep}
Let $p$ be an odd prime and let $K$ be a finite exceptional quasisimple group and 
$B\in \Blk(K)$ with defect group $D\cong C_p\times C_p$ and Sylow $B$-subgroup $(D, b_D)$. Then one of the following holds.

$(a)$ $a=1$.

$(b)$ $p=5$, $K=E_8(q)$, $e=1$ and $B=\cE_5(K, (s))$, where 
$s\in \cf^{5'}$ such that $C_{K^*}(s)=C_{q^4+\epsilon q^3+q^2+\epsilon q+1}\times C_{q^4+\epsilon q^3+q^2+\epsilon q+1}$.

\medskip
$(c)$ $p=3$, $K\in \{{}^3 D_4(q), F_4(q),(E_6)_u^{-\epsilon}(q),(E_7)_u(q)/Z, E_8(q)\}$
with $Z\leq Z((E_7)_u(q))$ such that 
$$
C_K(D)^\circ=\left\{ \begin{array}{lll}
C_{q^2+\epsilon q+1}\times C_{q^2+\epsilon q+1} & {\rm {if}} & K\in \{ {}^3D_4(q), F_4(q)\},\\
C_{q^2+\epsilon q+1}\times C_{q^4+q^2+1} & 
{\rm {if}} & K=(E_6)_u^{-\epsilon}(q), \\
C_{q^2+\epsilon q+1}\times C_{q^2+\epsilon q+1}\times (\SL_2(q^3)/Z) & {\rm {if}} & K=(E_7)_{u}(q)/Z, \\
C_{q^2+\epsilon q+1}\times C_{q^2+\epsilon q+1}\times {}^3D_4(q) & {\rm {if}} & K=E_8(q), \end{array} \right.
$$
and $b_D$ covers a block $b\in \Blk(C_K(D)^\circ)$ with $D(b)=D$, 
where $C_K(D)^\circ=(C_{\overline {K}}(D)^\circ)^\sigma=C_{C_{\overline {K}}(D)^\circ}(\sigma)$ and $C_K(D)=C_K(D)^\circ.p$ except when $K={}^3D_4(q)$, in which case $C_K(D)=C_K(D)^\circ$.

\end{proposition}

\begin{proof}
Let $K_u$ be a universal group such that $K=K_u/Z$ for some $Z\leq Z(K_u)$, and 
let $B_u\in \Blk(K_u)$ dominate $B$ and $D_u=D(B_u)$ such that $D=D_uZ/Z$.
If $Z(K_u)\neq \Omega_1(Z(D_u))$, then take $z\in Z(D_u)\setminus Z(K_u)$
with $|z|=p$. If $Z(K_u)=\Omega_1(Z(D_u))$, then take $z\in D_u$ such that
$zZ(K_u)\in Z(D_u/Z(K_u))$. Let $(z,B_z)$ be a
$B_u$-subsection, and suppose it is major when $z\in Z(D_u)$. In any case, let
$D_z$ be a defect group for $B_z$ with $D_z \leq D_u$. Then $B_z$
is a block of $C=C_{K_u}(z)$, and $D_u$ is a defect group for $B_z$ when $(z,B_z)$
is major. By \cite[Theorem 4.2.2]{gls98},
\[
C=O^{r'}(C)T,\quad O^{r'}(C)=L_1\circ L_2 \circ\cdots \circ L_\ell
\]
where each $L_i\in\Lie(r)$, and $T$ is an abelian $r'$-group inducing
inner-diagonal automorphisms on each $L_i$. In general, $z\not\in O^{r'}(C)$,
so we perform the following modifications. If $Z(C)\leq O^{r'}(C)$, then let
$\ell=k$ and $L=O^{r'}(C)$. If $Z(C)\not\leq O^{r'}(C)$, then let $k=\ell+1$
and $L_k=Z(C)$. Thus
$$
C=LT\quad {\rm with}\quad \quad L:=L_1\circ L_2 \circ\cdots \circ L_k,
\eqno(\ref{egps}.1)
$$
$z\in Z(C)\leq L$ and  $L\lhd C$. Let $B_L$ be a block of $L$ covered
by $B_z$. Note that $B_L$ has defect group
$D_L = D_z \cap L$. If $\chi\in \Irr(B_L)$, then $\chi=\chi_1\circ\cdots\circ\chi_k$ for
some $\chi_i\in\Irr(L_i)$, so that $\chi_i\in\Irr(B_i)$ for some
$B_i\in\Blk(L_i)$ and we write $B_L=B_1\circ B_2\circ\cdots\circ B_k$.

Since $L=(L_1\times L_2\times\cdots\times L_k)/A$ for a central subgroup
$A\leq (L_1\times L_2\times\cdots\times L_k)$, it follows by
\cite[Theorems 5.8.8 and 5.8.10]{NT} that
\[
D_L=(D(B_1)\times D(B_2)\times \cdots\times D(B_k))A/A
\] where $D(B_i)$ is some defect group of $B_i$ and $D(B_i)$ is isomorphic to
a subgroup of $D_L$.

Since $z\in K_u$, it follows by \cite[Theorem 4.1.9]{gls98} that $z$ is of parabolic or
exceptional type. Suppose $z$ is of parabolic type. Then 
${\overline {C}}=C_{{\overline {K}}_u}(z)=S {\overline {L}}$, where 
${\overline {K}_u}$ is the corresponding algebraic group, ${\overline {L}}=[{\overline {C}}, {\overline {C}}]$ is semisimple and $S=Z({\overline {C}})$ is a torue. Since $z$ is parabolic, we have that $\dim S\geq 1$ and hence $C_{q-\epsilon}\leq C_S(\sigma)\leq Z(C)$ except when
$p=3$ and $K=K_u={}^3 D_4(q)$ and $C=(C_{q^2+\epsilon q+1}\circ_3 (\SL_3^\epsilon (q).3))$, 
in which case $Z(C)=C_{q^2+\epsilon q+1}$ (see \cite[Table 4.7.3A]{gls98}).

Suppose $(K_u,p)\neq ((E_6)_u^\epsilon(q), 3)$ and  $C_{q-\epsilon}\leq Z(C)$.
Since $Z(C)$ is abelian and a defect group of $L$ is radical, we have
$C_{p^a}\leq D_L\leq D_u$. But $C_{p^a}\cap Z=1$, so $C_{p^a}Z/Z=C_{p^a}\leq D$. Hence $a=1$.

\bigskip
Suppose $z$ is exceptional, so that $p$ is not good for $K$ in the sense of \cite[Table 1]{CE94}. Thus $p\in \{3, 5\}$ and moreover, $K=K_u=E_8(q)$ when $p=5$. In particular,
(a) holds when $p\geq 7$. 

Suppose $D\cong D_u$ and $(K_u,p)\neq ((E_6)_u^\epsilon(q), 3)$. If some $y\in D_u\setminus \{1\}$ is not exceptional, then replace $z$ by 
$y$ and repeat the proof above. We may suppose each $y\in D_u\setminus \{1\}$ is exceptional, except when $K=K_u={}^3 D_4(q)$, in which case each $x\in D_u\setminus \{1\}$ satisfies
$C_K(x)\cong (C_{q^2+\epsilon q+1}\circ_3 (\SL_3^\epsilon (q).3))$. Denote by $T_{2,\epsilon}=C_{(q^{ep}-\epsilon)/(q^e-\epsilon)}\times C_{(q^{ep}-\epsilon)/(q^e-\epsilon)}$.

\bigskip
(b) If $p=5$,  then $K=K_u=E_8(q)$ and by \cite[Table 4.7.3B]{gls98},
\[
(e, C)\in \{(1, (\SL_5^\epsilon (q)\circ_5 \SL_5^\epsilon (q)).x), (2, \SU_5(q^2).5)\},
\]where $x=5{:}5$. Since $D_u\leq C$, we have $e=1$ and $C=(\SL_5^\epsilon (q)\circ_5 \SL_5^\epsilon (q)).x$. In addition, $D=\langle z, y\rangle$ for some 
$y\in C\setminus  (\SL_5^\epsilon (q)\circ_5 \SL_5^\epsilon (q))$ as $y$ is exceptional. So 
\[
C_K(D)=C_K(D)^\circ.5=T_{2,\epsilon}.5=(C_{q^4+\epsilon q^3 +q^2+\epsilon q+1})^2.5.
\]
By \cite[Table 4.7.1]{gls98}, $\oC=\SL_5(\bF)\circ_5\SL_5(\bF)$ is a connected reductive group, and by \cite{FS82}, $B_z=\cE_5(C, (t))$ for some $t\in \oC^*$, where $\bF$ is the algebraic closure of $\bF_q$.
In addition, $D=D(B_z)$ is isomorphic to a Sylow subgroup of $C_{C^*}(t)$. Hence $C_{\oC^*}(t)=\oT^*$ is a maximal
torus such that $\oT^\sigma\cong T_{2,\epsilon}$, where $\oT$ is the dual of $\oT^*$.

Take $\oH_1=\GL_5(\bF)=\oT_1\circ_5 \SL_5(\bF)$ and $\oH_2=\SL_5(\bF)\circ_5\oT_1$ be subgroups of $\oC$. 
Then each $\oH_i$ is a Levi subgroup of both $\oC$ and $\oK$, and $C_{\oC^*}(t)\leq \oH_i^*$.
If $H_i=\oH_i^\sigma$, then we may suppose $t\in H_i^*\leq K^*$. Since each $C_{H_i^*}(t)=T_{2,\epsilon}^*$, it follows that
$C_{K^*}(t)=T_{2,\epsilon}^*$. If $C_{K^*}(t)=T_{2,\epsilon}^*$, then 
$b=\cE_p(T_{2,\epsilon}, (t))$ is a block with a defect group $O_p(T_{2,\epsilon})=C_p\times C_p$.
But $T_{2,\epsilon}$ is a Levi subgroup and $R_{T_{2,\epsilon}}^K$ is a perfect isometry,
so $B=\cE_p(K, (t))$ is a block and $D(B)\cong C_p\times C_p$. Here $R_{T_{2,\epsilon}}^K$ is the Deligne-Lusztig map from 
$\cE_p(T_{2,\epsilon}, (t))$ to $\cE_p(K, (t))$. Similarly, $R_{T_{2,\epsilon}}^C(b)=B_z$.
Thus (b) holds.

\bigskip
(c) Suppose $p=3$, so that $e=1$. The proof is similar to that of (b).
Let $K=K_u={}^3 D_4(q)$, so that $\oC=\oT_2\circ_3\SL_3(\bF)$ is Levi. Thus $C=C_{q^2+\epsilon q+1}\circ_3 (\SL_3^\epsilon (q).3)$ and $B_z=\cE_p(C, (s))$ for some $s\in (\oC^*)^\sigma$. So $D=D_u=C_p\times C_p$ and $D=\langle z, y\rangle$ for some 
$y\in (\SL_3^\epsilon (q).3)\setminus \SL_3^\epsilon (q)$. Thus
$C_K(D)=C_C(y)=T_{2,\epsilon}=(C_{q^2+\epsilon q+1})^2$ and $s\in  T_{2,\epsilon}^*$.
A proof similar to that of (b) shows that $C_{K^*}(s)=T_{2,\epsilon}^*$ (cf. \cite[Table 1]{An95}). Conversely, if $s\in K^*$ such that $C_{K^*}(s)=T_{2,\epsilon}^*$, then
$B=\cE_p(K, (s))$ is a block and $D(B)=O_p(T_{2,\epsilon})\cong C_p\times C_p$.
Thus (c) holds.

\medskip
Suppose $K=F_4(q)$, so that $K=K_u$, $\oC=\SL_3(\bF)\circ_3\SL_3(\bF)$, $C=(\SL_3^\epsilon(q)\circ_3 \SL_3^\epsilon(q)).x$ and $D=D_u=C_p\times C_p$, where $x=3{:}3$. A proof similar to that of (b) shows that 
$C_K(D)=T_{2,\epsilon}.3=(C_{q^2+\epsilon q+1})^2.3$, $B_z=\cE_p(C, (s))$ and 
$s\in T_{2,\epsilon}^*$. In addition, $C_{K^*}(s)=T_{2,\epsilon}^*$ and $B=\cE_p(K, (s))$ is a block
with defect group $D=O_p(T_{2,\epsilon})=C_p\times C_p$.

\bigskip
If $K=K_u=(E_{6})_{u}^{-\epsilon}(q)$ and $z$ is exceptional, then 
$\oC=(\SL_3(\bF))^3/3$ and
by \cite[Table 4.7.3A]{gls98}, $C=(\SL_3^\epsilon(q)\circ_3 \SL_3(q^2)).x$ with $x=3{:}3$. Thus 
$D=\langle z, y\rangle$ for some $y\in C\setminus \SL_3^\epsilon(q)\circ_3 \SL_3(q^2)$, and $C_K(D)=C_K(D)^\circ .3=(C_{q^2+\epsilon q+1}\times C_{q^4+q^2+1}).3$. A proof similar to that of (b) shows 
that $B_z=\cE_p(C, (s))$, $s\in (C_{q^2+\epsilon q+1}\times C_{q^4+q^2+1})^*$ and  $C_{K^*}(s)\cong 
C_{q^2+\epsilon q+1}\times C_{q^4+q^2+1}$. Conversely, if $C_{K^*}(s)\cong C_{q^2+\epsilon q+1}\times C_{q^4+q^2+1}$, then $B=\cE_p(K, (s))$ is a block with defect group $D=O_p(C_{q^2+\epsilon q+1}\times C_{q^4+q^2+1})=C_p\times C_p$.

\bigskip
If $K_u=(E_7)_u(q)$, then by \cite[Table 4.7.3A]{gls98}, $C=(\SL_3^\epsilon(q)\circ_{6^*}\SL_6^\epsilon(q).6^*).x$, where $6^*=\gcd(6, q-\epsilon)$ and $x=3{:}3$. 
Note that $D=D_uZ/Z\cong D_u$. So $D=\langle z, y\rangle$ for some $y\in C\setminus (\SL_3^\epsilon(q)\circ_{6^*}\SL_6^\epsilon(q).6^*)$. Hence $C_{K_u}(D)=
C_{K_u}(D)^\circ.p=(T_{2,\epsilon}\times \SL_2(q^3)).p$. If $(D, b_D)$ is a Sylow $B$-subgroup, then $b_D$ covers a block $b=b_1\times b_2\in \Blk(C_{K_u}(D)^\circ)$, where $b_1\in \Blk(T_{2,\epsilon})$ with $D(b_1)=D$ and 
$b_2\in\Blk(\SL_2(q^3))$ with $D(b_2)=1$.
Conversely, suppose $b=b_1\times b_2\in \Blk(C_{K_u}(D)^\circ)$ such that
$b_1\in \Blk(T_{2,\epsilon})$ with $D(b_1)=O_p(T_{2,\epsilon})=D$ and
$b_2\in \Blk(\SL_2(q^3))$ with $D(b_2)=1$. Note that $C_{K_u}(D)^\circ$ is a Levi subgroup of $K_u$, so $B=R_{C_{K_u}(D)^\circ}^{K_u}(b)$ is a block of $K_u$ with $D(B)=D$.
Thus (c) holds. If $Z=Z(K_u)\neq 1$, then $Z=C_2$, $C_K(D)=C_{K_u}(D_u)/Z=(T_{2,\epsilon}\times \SL_2(q^3)/Z).p$. A proof similar to that of $K_u=(E_7)_u(q)$ shows that (c) also holds. 

\bigskip
If $K=K_u=E_8(q)$, then  by \cite[Table 4.7.3A]{gls98}, 
\[
C\in  \{(\SL_9^\epsilon(q)/3).(9^*/3), (\SL_3^\epsilon(q)\circ_3 (E_6)_u(q)).x\}
\]
with $x=3{:}3$. By \cite[Table 4.7.3A]{gls98} again, $C_K(D)=C_K(D)^\circ.p =
(T_{2,\epsilon}\times {}^3D_4(q)).p$. A proof similar to that of $K_u=(E_7)_u(q)$ shows that the root block $b_D\in \Blk(C_K(D))$ of $B$ covers a block 
$b=b_1\times b_2\in \Blk(C_K(D)^\circ)$, where $b_1\in\Blk(T_{2,\epsilon})$ with $D(b_1)=D$ and $b_2\in\Blk({}^3D_4(q))$ with $D(b_2)=1$. 
Conversely, if $b\in  \Blk(C_K(D)^\circ)$ with $D(b)=D$, then $B=R_{C_K(D)^\circ}^K(b)$ has a defect group $D$. So (c) holds.

\bigskip
If $K=K_u=(E_{6})_{u}^\epsilon(q)$ and  $z$ is parabolic, then by \cite[Table 4.7.3A]{gls98}, $C_{q-\epsilon}\leq Z(C)$ and
$C_{p^a}\leq D$ as $D$ is a defect group of $C$. Hence $a=1$ and (a) holds. 
If $z$ is exceptional, then $C=((\SL_3^\epsilon(q) \times \SL_3^\epsilon(q)\times \SL_3^\epsilon(q))/3).x$ such that $x=3{:}3{:}3$. Since $D\leq C$ is a defect group, we have
$Z(C)=C_p\times C_p\leq D$ and hence $D=Z(C)$. Thus 
$C_{K_u} (D)=C$.  So $L=(\SL_3^\epsilon(q) \times \SL_3^\epsilon(q)\times \SL_3^\epsilon(q))/3$ and $B_L=B_1\circ B_2\circ  B_3$ with  $D(B_L)=
D(B_1)\circ D(B_2)\circ D(B_3)=D$. Each $B_i\in \Blk(\SL_3^\epsilon(q))$ with
$D(B_i)=C_3=Z(\SL_3^\epsilon(q))$. Since $x=3{:}3{:}3$ induces inner-diagonal automorphism
on each factor $\SL_3^\epsilon(q)$, we have that $x$ stabilises each $B_i$ and hence
$x$ stabilises $B_L$. If $(D, b_D)$ is a Sylow $B$-subgroup such that $b_D$ covers $B_L$, then $D\neq D(b_D)$, which is impossible.  

\bigskip
If $D\not \cong D_u$, then $(K, p)=((E_6)_a^\epsilon(q), 3)$ and $Z(K_u)=C_3\leq D_u$.
In this case, we take $z\in D$ and let $(z, B_z)$ be a major subsection of $B$. 
Hence $D=D(B_z)$ and $D\leq C_K(z)$ is a defect group. If $z$ is of parabolic type, then
by \cite[Table 4.7.3A]{gls98}, $C_{q-\epsilon}\leq Z(C_K(z))$ and hence 
$C_{p^a}\leq D$ as $D$ is radical of $C_K(z)$. In particular, $a=1$. Note that if 
$C_K(u)\cong C_{q^2+\epsilon q+1}\times {}^3D_4(q)$ for some order $3$ element 
$u\in K$ given in \cite[Table 4.7.3A]{gls98}, then $u$ is diagonal, which is impossible
as $K$ is quasisimple. If $z$ is of exceptional type, then
$C_K(z)=(\SL_3^\epsilon(q)\times \SL_3^\epsilon(q)\times \SL_3^\epsilon(q))/(3^2).x.t$, where $x=3{:}3{:}1$ and $t=1{:}3{:}3$. Thus $D\leq C_K(z)$ contains always an order $3$ element  of parabolic type. Hence (a) holds.
\end{proof}

\begin{remark}\label{Rem:excep}
$(a)$ Let $p=5$ and $K_0=E_8(r)\leq K=E_8(q)$ such that $r\equiv 1$ {\rm mod} $5$
and $q=r^f$ for an integer $f$. Take $s\in (K_0)^*$ such that 
$C_{(K_0)^*}(s)=C_{(r^5-1)/(r-1)}\times C_{(r^5-1)/(r-1)}$ and $C_{K^*}(s)=C_{(q^5-1)/(q-1)}\times C_{(q^5-1)/(q-1)}$.  Then $B=\cE_5(K, (s))$ is a $5$-block with defect group
$C_5\times  C_5$. Since $s\in (K_0)^*$, it follows that $\sigma$ stabilises $B$.
In particular, if $p\mid f$, then $B$ is stabilised by a field automorphism of order $p$.

Similarly we can construct $3$-blocks of some finite exceptional groups that are stabilised 
by a field automorphism of order $3$.

$(b)$ Suppose $a=1$. Then $p\nmid f$, and $K$ has no field automorphism of order $p$ except when $p=3$ and $K={}^3D_4(q)$, in which case the field automorphism group of $K$ is cyclic of order $3f$.  

$(c)$ Let $p=3$, $K={}^3D_4(q)$, and let $D\leq K$ be a defect group given by Proposition $\ref{Prop:excep}$ $(c)$ and $Q\leq K$ a subgroup such that $Q\cong C_p\times C_p$ but $Q\neq_K D$. By \cite[Table 4.7.3A]{gls98}, $K$ has two order $3$ conjugacy classes, so there is $x\in Q$ such that 
$C_K(x)=(q-\epsilon)\circ_{2^*} (\SL_2(q^3).2^*)$. Thus $Q\leq C_K(x)$,
$Q=\langle x, y\rangle$ with $y\in \SL_2(q^3)$ and $C_K(Q)=C_{q-\epsilon}\times C_{q^3-\epsilon}$. In particular, $Q$ is not radical. Therefore, $D$ is the only defect group of $K$ isomorphic to $C_p\times C_p$. 
\end{remark}

Recall $K=O^{r'}(C_{\oK}(\sigma))$, so $K\neq C_{\oK}(\sigma)$ in general. 
But if $K=K_u$, then $K=C_{\oK}(\sigma)$. Note that if  $K$ is given by Proposition \ref{Prop:excep} (b) or (c), then  $K=K_u$ except when $K=(E_7)_{a}(q)=(E_7)_{u}(q)/Z$, 
in which case $Z=Z((E_7)_{u}(q))=C_2$. Since $p$ is odd and $Z$ is a $p'$-group, we may always suppose $K=K_u$, and hence $K=C_{\oK}(\sigma)$. 

Let $B$ be a block of $K$ with a defect group $D$ which is given by Proposition \ref{Prop:excep} (b) or (c), then
$K=K_u$ and hence $D$ is toral. In the following we are going to show that if $B$ is stabilised by a field automorphism $\sigma'$ of order $p$, then  $D$ is centralised by $\sigma'$.

Recall that $q=r^f$ and $e'$ is the multiplicative order of $\epsilon r$ modulo $p$
and  $e$ is the multiplicative order of $\epsilon q$ modulo $p$. Suppose $e=1$.
If $r\equiv \pm 1\pmod 5$, then $e'=1$ and $e'\mid f$. If $r\equiv \pm 2\pmod 5$, then
$r^2 \equiv -1 \pmod 5$, $e'=2$ and $|r|=4$ in $C_5$.  But $r^{2f}\equiv 1 \pmod 5$, so
$4\mid 2f$ and $e'=2\mid f$. Write $f=e' f_1$ and $q_0=r^{e'}$. If $p=3$, then
$e=e'=1$, we set $q_0=r$. Let $K_0\leq K$ be the subgroup of $K$ defined over $q_0$. So 
\[
K_0\in \{E_{8}(q_0), {}^3D_4(q_0), F_4(q_0), (E_6)_u^{-\epsilon}(q_0), (E_7)_u(q_0)\}
\]
with $K_0\leq K$. Denote by $(\oK, \sigma_0)$ a $\sigma$-setup for $K_0$, so
$K_0=C_{\oK}(\sigma_0)$ and $\sigma=\sigma_0^{f_1}$. 

Let $\ell$ be the Lie rank of $\oK$. Take a maximal torus
$T_\epsilon=(q_0-\epsilon)^\ell$ of $K_0$ and so $E=\Omega_1(O_p(T_\epsilon))=(C_p)^\ell$. Thus $C_{K_0}(E)=T_\epsilon$, $C_K(E)=T_K\leq K$ and $C_{\oK}(E)=\oT$ are maximal tori of $K$ and $\oK$, respectively. Note that \[
E=\Omega_1(O_p(T_K))=\oT_{(p)}:=\{ x\in\oT  \mid x^p=1\}.
\] 
Let $V$ be the extended Weyl group of $\oK$, so that $N_{\oK}(\oT)=\langle \oT, V\rangle$, $V/(\oT \cap V)$ is a Weyl group and $V$ is centralised by all field automorphisms of $\oK$. In particular, $V\leq K_0$.

\begin{lemma}
\label{Lem:stabfieldaut}
Let $p$ be an odd prime and follow the notation above. Let $K$ be a quasisimple group
and $B\in \Blk(K)$ with defect group $D=C_p\times C_p$ given in Proposition $\ref{Prop:excep}$ $(b)$ or $(c)$. Then $K_0$ contains a $K$-conjugate of $D$.
\end{lemma}
   
\begin{proof}
Let $X \leq K_0$ be a toral elementary abelian $p$-subgroup. By Proposition 
\ref{prop:eleoK}, $X=A_w$ is $\oK$-conjugate to some $A\leq \oT$, where $w\in C_{\oK}(A)/C_{\oK}(A)^\circ$. Note that $A\leq T_{(p)}=E$, so $A\leq K_0$. If 
$x\in N_{\oK}(A)$, then $\oT, \oT^x \leq C_{\oG}(A)^\circ$, so $\oT^{xy}=\oT$ for some $y\in C_{\oG}(A)^\circ$. Thus $xy =vt\in N_{\oK}(\oT)$ with  $v\in V, t\in \oT$. Hence $x= v u$ for some $u\in C_{\oG}(A)^\circ$, 
$N_{\oK}(A)=C_{\oK}(A)^\circ N_V(A)$ and $C_{\oK}(A)=C_{\oK}(A)^\circ C_V(A)$.
Since $V\leq K_0$, we have that $\sigma_0$ induces identity map on 
$N_{\oK}(A)/C_{\oK}(A)^\circ$ and $C_{\oK}(A)/C_{\oK}(A)^\circ$.
In particular, a $\sigma_0$-class of $N_{\oK}(A)/C_{\oK}(A)^\circ$ contained in $C_{\oK}(A)/C_{\oK}(A)^\circ$ is a $N_{\oK}(A)/C_{\oK}(A)^\circ$-conjugacy class contained in 
 $C_{\oK}(A)/C_{\oK}(A)^\circ$. 
 
Repeat the proof above with $X\leq K_0$ replaced by $X\leq K$. We have that $X=A_w$ with
$w\in C_{\oK}(A)/C_{\oK}(A)^\circ$, $A\leq T_K$, and $\sigma$ induces identity map on 
$N_{\oK}(A)/C_{\oK}(A)^\circ$ and $C_{\oK}(A)/C_{\oK}(A)^\circ$.
Since $N_{\oK}(A)/C_{\oK}(A)^\circ$ and $C_{\oK}(A)/C_{\oK}(A)^\circ$ are independent 
of $q_0$ and $q$, we have that  $K$ and $K_0$ have the same number of 
conjugacy classes of toral elementary abelian $p$-subgroups. In particular, the 
$K$-conjugacy class of $D$ contains a representative in $K_0$.  
\end{proof}

\begin{corollary}
\label{Cor:stabfieldaut}
Let $p$ be an odd prime and let $K$ be a quasisimple group
and $B\in \Blk(K)$ with defect group $D=C_p\times C_p$ given by Proposition $\ref{Prop:excep}$ $(b)$ or $(c)$. Suppose $\sigma'$ is a field automorphism of $K$ of order $p$. Then we may suppose $\sigma'$ centralises $D$.
\end{corollary}

\begin{proof}
Follow the notations of Lemma \ref{Lem:stabfieldaut}. We may suppose
$D\leq K_0$. Note that $f=e'f_1$ with $e'\in \{1,2\}$. If $p\mid f$, then $p\mid f_1$ and so $f_1=pf_2$ for some integer $f_2$. We may suppose $\sigma'=\sigma_0^{f_2}$ and hence $\sigma'$ centralizes $K_0=C_{\oK}(\sigma_0)$, and in particular, $\sigma'$ centralizes $D$.

Suppose $p\nmid f$. If $K\neq {}^3D_4(q)$, then $K$ has no field automorphism of order $p$. Suppose $K={}^3D_4(q)$, so field automorphism group of $K$ is cyclic of order $3f$,
and an order $3$ field automorphism is also an order $3$ graph automorphism.
By \cite[Table 4.7.3A]{gls98}), $K$ has two order $3$ field automorphisms $\gamma_1, \gamma_2$ 
such that $C_K(\gamma_1)=G_2(q)$ and $C_K(\gamma_2)={\rm PSL}_3^\epsilon(q).3$. In addition,
$\gamma_1=h \gamma_2$ for some $h\in K$. So $\gamma_1$ centralises a $K$-conjugate of $D$ if 
and only if $\gamma_2$ centralises a $K$-conjugate of $D$. 
Take $E=C_p\times C_p\leq {\rm PSL}_3^\epsilon(q)$. Then $E$ is Sylow and
$\Out_{{\rm PSL}_3^\epsilon(q)}(E)=Q_8$. If $Q$ is given in Remark \ref{Rem:excep} (c),
then $\Out_K(Q)=D_{12}$ (cf. \cite[(1A)]{An95}), so $E\neq _K Q$ and hence $E=_K D$. In particular, $\sigma'$ centralises a $K$-conjugate of $D$.  
\end{proof}

\begin{lemma}
\label{Lem:stabfieldautp}
Let $p$ be an odd prime and let $K$ be a quasisimple group of Lie type
and $B\in \Blk(K)$ with a defect group $D=C_p\times C_p$ given in Proposition $\ref{Prop:excep}$ $(b)$ or $(c)$. Suppose
$K\leq H$  such that $[H{:}K]=p$ and $H/K\leq \Out(K)$.
Let $B_H$ be a block of $H$ covering $B$.
Then a defect group of $B_H$ is abelian of order $p^2$ or $p^3$.
\end{lemma}

\begin{proof}
Let $(D, b_D)$ be a Sylow $B$-subgroup. By \cite[Lemma 4.1]{E05}, there is a $B_H$-subgroup $(D, b_H)$ such that $b_H$ covers $b_D$. 

\medskip
(a) First we show that $C_H(D)$ contains a defect group $D_H$ of $B_H$.

We may suppose $D\leq D_H$ for some $D_H=D(B_H)$ and so $|D_H{:}D|\leq p$. If $D=D_H$, then  Lemma \ref{Lem:stabfieldautp} holds. Suppose $|D_H{:}D|=p$.
Note that a Sylow $p$-subgroup of $\Out(K)$ is cyclic, so by
Corollary \ref{Cor:stabfieldaut}, we may suppose $H=\langle K,\sigma'\rangle$, where $\sigma'$ is an order $p$ field automorphism such that $[D,\sigma']=1$. Thus $\sigma'\in C_H(D)$,
\[
C_H(D)=\langle C_K(D), \sigma'\rangle
\] 
and $C_H(D)/C_K(D)\cong C_p$. Note that if $y\in N_H(D, b_H)$, then
$y$ stabilises $b_D$, so $N_H(D, b_H)\leq N_H(D, b_D)$. 
But $\sigma'\in C_H(D)$, so $\sigma'$ stabilizes $b_D$. Since 
\[
N_H(D, b_D)/N_K(D, b_D)\leq H/K,
\]
we have 
\[
N_H(D, b_D)=\langle N_K(D, b_D),\sigma'\rangle=C_H(D) N_K(D, b_D).
\]  
So $N_H(D, b_D)/C_H(D)\cong N_K(D, b_D)/C_K(D)$ is a $p'$-group, hence
$N_H(D, b_H)/C_H(D)$ is a $p'$-group. 

Since $|D_H{:}D|=p$, we have that $(D, b_H)$ is not a Sylow $B_H$-subgroup
and hence $|N_H(D, b_H){:} D|_p=p$. But $N_H(D, b_H)/C_H(D)$ is a $p'$-group, so $C_H(D)$ contains a defect group, say $D_H$ of $B_H$. 

\medskip
Suppose $(K, p)=({}^3D_4(q), 3)$. Then $C_K(D)=T_{2,\epsilon}\cong (C_{q^2+\epsilon q+1})^2$ as shown in the proof of  Proposition \ref{Prop:excep} (c). Thus
$C_H(D)=T_{2,\epsilon}.\sigma'$ and $\langle D, \sigma'\rangle$ is an abelian Sylow $p$ subgroup
of $C_H(D)$, Hence $D_H\leq C_H(D)$ is abelian. 

\medskip
Suppose 
\[
(K, p)\in \{(E_8(q), 5), (F_4(q),3), ((E_6)_u^{-\epsilon}(q), 3), (E_8(q), 3), ((E_7)_u(q)/Z, 3)\}.
\] 
As shown in the proof of Proposition \ref{Prop:excep} (b) or (c), 
\[
C_K(D)=C_K(D)^\circ.w=C_K(D)^\circ.p,
\]
where $C_{K}(D)^\circ= (C_{\oK}(D)^\circ)^\sigma$ and $w$ induces an order $p$ action on 
$C_K(D)^\circ$. Suppose $K=\oK^\sigma$. Note that $\oK$ is simply-connected (or universal), 
so every elementary abelian subgroup of size $p^2$ is a subgroup of a maximal torus.
We may suppose $D\leq \oT\leq \oK$ for some maximal torus $\oT$. So
$\oT\leq C_{\oK}(D)^\circ$. Since $w\not\in C_K(D)^\circ=(C_{\oK}(D)^\circ)^\sigma$, we have
$w\in C_{\oK}(D)\setminus C_{\oK}(D)^\circ$.

\medskip
(b) Next we show that there is an element $v\in C_K(D)\setminus C_{K}(D)^\circ$ such that $[v,\sigma']=1$.
(This is well-konwn and proved for all toral elementary abealin $p$-subgroups of $\oK$ in \cite{ADL}, and we give a proof here for completeness).

Since  $w\in C_{\oK}(D)\setminus C_{\oK}(D)^\circ$, we have 
$\oT$ and $\oT^w$ are two maximal tori of $ C_{\oK}(D)^\circ$, and uniqueness implies that
$\oT^{wc}=\oT$ for some $c\in C_{\oK}(D)^\circ$. Thus 
$wc\in N_{\oK}(\oT)=\oT V$ and $wc t\in V\leq K_0$ for some $t\in \oT$, where $V$ is the extended Weyl group (cf. Lemma \ref{Lem:stabfieldaut}). Set $v=wct$, so 
$v\in C_K(D)$ and $v\not\in C_K(D)^\circ$ as $v\not\in C_{\oK}(D)^\circ$. 
Since $C_K(D)/C_{K}(D)^\circ\cong C_p$, we have $C_K(D)=C_K(D)^\circ.v$.

\medskip
(c) Now we show that $D_H$ is abelian.

Suppose 
\[
(K, p)\in \{(E_8(q), 5), (F_4(q),3), ((E_6)_u^{-\epsilon}(q), 3)\}.
\]
As shown in the proof of of Proposition \ref{Prop:excep} (b) or (c), 
$C_K(D)\cong (C_{q^4+\epsilon q^3+q^2+\epsilon q+1})^2.5$
when $(K, p)\in (E_8(q), 5)$; $C_K(D)\cong (C_{q^2+\epsilon q+1})^2.3$ when 
$(K, p)\in (F_4(q), 3)$; $C_K(D)=(C_{q^2+\epsilon q+1}\times C_{q^4+q^2+1}).3$
when $(K, p)=((E_6)_u^{-\epsilon}(q), 3)$. Thus 
$C_K(D)=T.p$, where $C_K(D)^\circ=T$ is a maximal torus of $K$. By the proof of part (b) above, 
$C_K(D)=\langle T, v\rangle$ for some $v\not\in T$ and $[v,\sigma']=1$.

Since $v^p\in T$ and $v\not\in T$, we have that the abelian subgroup $\langle D, v, \sigma'\rangle$ contains a Sylow $p$-subgroup of $C_H(D)$. Hence any Sylow
$p$-subgroup of $C_H(D)$ is abelian, so is $D_H$.

\medskip
Suppose  $(K, p)\in \{(E_8(q), 3), ((E_7)_u(q)/Z, 3)\}$.  By Proposition \ref{Prop:excep} (c), 
\[
C_K(D)=(C_{q^2+\epsilon q+1}\times  C_{q^2+\epsilon q+1}\times L).v\quad 
\hbox{ and } \quad C_H(D)=\langle C_K(D),\sigma'\rangle,
\]
where $v$ acts as an order $p$ element on $C_K(D)^\circ$, $[v,\sigma']=1$ and 
$L=(\SL_2(q^3)/Z)$ or ${}^3D_4(q)$ according as $K=(E_7)_u(q)/Z$ or $E_8(q)$. Note that 
$L$ is a normal subgroup of $C_H(D)$. If $D_H\cap L\neq 1$, then $(D_H\cap L)\cong C_p$
and $D_H=D\times (D_H\cap L)$ is abelian. Suppose $D_H\cap L=1$, so that
$D_H\cong (D_HL)/L\leq C_H(D)/L=\langle C_K(D)/L,\sigma'\rangle$. 
Note that the abelian subgroup  $\langle D, v, \sigma'\rangle L/L$ contains a Sylow $p$-subgroup of $C_H(D)/L$, hence any Sylow $p$-subgroup of $C_H(D)/L$ is abelian, so is $D_H$. 

\end{proof}


\section{Reductions}
\label{reductions:section}

The first step of our reduction is to apply Fong-Reynolds and K\"ulshammer-Puig reductions as far as possible, meaning that we may establish a Morita equivalence with what we will call a reduced block. A feature of these particular Morita equivalences is that they preserve both fusion and the K\"ulshammer-Puig classes as reviewed in Section \ref{background}. The proof of this part of the reduction is based on the first part of the proof of~\cite[Proposition 4.3]{eel20}. Once we have established that it suffices to consider reduced blocks, we will then analyse their structure.

\begin{proposition}
\label{reduce:prop}
Let $G$ be a finite group and let $B$ be a block of $\cO G$ with defect group $D$. There is a finite group $H$ and a block $C$ of $\cO H$ with defect group isomorphic to $D$ such that $B$ is basic Morita equivalent to $C$ and the following are satisfied:
\begin{enumerate}
\item[(R1)] $C$ is quasiprimitive, that is, every covered block of a normal subgroup of $H$ is $H$-stable;
\item[(R2)] If $N \lhd H$ and $C$ covers a nilpotent block of $\cO N$, then $N \leq Z(H) O_p(H)$ and $O_{p'}(N) \leq Z(H) \cap H'$.

The block $C$ has the same fusion system $\mathcal{F}$ as $B$, and the same K\"ulshammer-Puig class in $H^2(\Aut_\mathcal{F}(D),k^\times)$.
\end{enumerate}
\end{proposition}

\begin{proof}
We use two process, labelled (FR) and (KP), which we apply alternately to blocks with defect groups isomorphic to $D$ and their normal subgroups. Each process strictly decreases $[G:O_{p'}(Z(G))]$ when applied nontrivially and so repeated application must come to an end, resulting in a block satisfying (R1) and (R2).

Let $G$ be a finite group and $B$ be a block of $\cO G$ with defect group $D$.

\begin{enumerate}
\item[(FR)] Suppose $N \lhd G$ and $b$ is a block of $\cO N$ covered by $B$. Write $I=I_G(b)$ for the stabilizer of $b$ in $G$, and $B_I$ for the Fong-Reynolds correspondent, which is the unique block of $I$ covering $b$ with Brauer correspondent $B$. Suppose that $I \neq G$. Now by~\cite[Theorem 6.8.13]{lin2} $B_I$ is source algebra equivalent to $B$. Hence by~\cite[Theorem 8.7.1]{lin2} $B_I$ has  defect group $D$ and the same fusion system, and by~\cite[Theorem 8.14.1]{lin2} it has the same K\"ulshammer-Puig class. Since $O_{p'}(Z(G)) \leq O_{p'}(Z(I))$ and $I \neq G$, we have $[I:O_{p'}(Z(I))]<[G:O_{p'}(Z(G))]$. Process (FR) involves replacing $B$ by $B_I$ and repeating the process for all normal subgroups until $B$ is quasiprimitive.

\item[(KP)] Assume that process (FR) has been performed, so that $B$ is quasiprimitive. Let $N \lhd G$ such that $N \not\leq (Z(G) \cap [G,G])O_p(G)$, and suppose that $B$ covers a nilpotent block $b$ of $N$. Let $b'$ be a block of $Z(G)N$ covered by $B$ and covering $b$. Since $B$ is quasiprimitive, both $b$ and $b'$ are $G$-stable. Further $b'$ must also be nilpotent. By~\cite[Theorem 8.12.5]{lin2} and~\cite[Remark 8.12.8]{lin2} $B$ is basic Morita equivalent to a block $\tilde{B}$ of a central extension $\tilde{L}$ of a finite group $L$ by a $p'$-group $\tilde{Z}$ (which we may further assume is contained in the derived subgroup $[\tilde{L},\tilde{L}]$) such that there is $\tilde{M} \lhd \tilde{L}$ with $\tilde{M} \cong D \cap (Z(G)N)$, $G/Z(G)N \cong \tilde{L}/\tilde{Z}\tilde{M}$, and $\tilde{B}$ has defect group isomorphic to $D$, the same fusion system as $B$, and the same K\"ulshammer-Puig class. We have $[\tilde{L}:O_{p'}(Z(\tilde{L}))] \leq |L| = [G:Z(G)N]|D \cap (Z(G)N)| < [G:O_{p'}(Z(G))]$ and $\tilde{M} \leq (Z(\tilde{L}) \cap [\tilde{L},\tilde{L}])O_p(\tilde{L})$. Process (KP) consists of replacing $G$ by $\tilde{L}$ and $B$ by $\tilde{B}$.
\end{enumerate}

Repeated application of (FR) and (KP) to all blocks of all normal subgroups must final result in a block to which we may no longer apply (FR) or (KP), that is, a block $C$ satisfying (R1) and (R2). 
\end{proof}

Call a pair $(G,B)$ satisfying conditions (R1)-(R2) of the proposition \emph{reduced}. If the group $G$ is clear, then just say that $B$ is reduced. 

Reduced blocks with defect group $D \cong p_+^{1+2}$ have a very restricted structure:

\begin{proposition}
\label{reduced_classification:prop}
Let $(G,B)$ be a reduced pair where $D$ has defect group $D \cong p_+^{1+2}$ with $p \geq 3$. 
Then either: (i) $B$ is a block of maximal defect and one of the following occurs:
\begin{enumerate}
\item[(S1)] $D \lhd G$;
\item[(S2)] $G \cong (C_p \times C_p) \rtimes H$, where $SL_2(p) \leq  H \leq GL_2(p)$;
\item[(S3)] $G$ satisfies $PSL_3(p) \leq G/Z(G) \leq \Aut(PSL_3(p))$, where $Z(G) \leq G'$;
\item[(S4)] $G$ satisfies $PSU_3(p) \leq G/Z(G) \leq \Aut(PSU_3(p))$, where $Z(G) \leq G'$;
\item[(S5)] $p=3$ and $G/Z(G)$ is one of $M_{12}$, $M_{12}.2$, $M_{24}$, $J_2$, $J_2.2$, $J_4$, $He$, $He.2$ or $Ru$;
\item[(S6)] $p=5$ and $G/Z(G)$ is one of $2.HS$, $2.HS.2$, $3.McL$, $McL.2$, $2.Ru$, $Co_3$, $Co_2$ or $Th$;
\item[(S7)] $p=11$ and $G=J_4$, and $B$ is the principal block;
\item[(S8)] $p=13$ and $G$ is the Fischer-Griess Monster $M$, and $B$ is the principal block;
\end{enumerate}

or (ii) $p=3$ and one of the following occurs:

\begin{enumerate}
\item[(S9)] $G$ is $Co_1$ or the Fischer-Griess Monster $M$;
\item[(S10)] $G$ satisfies $PSL_{3\delta}(q) \leq G/Z(G) \leq \Aut(PSL_{3\delta}(q))$ with $|Z(G)|_3=3$, where $3 \parallel q-1$ and $Z(G) \leq G'$;
\item[(S11)] $G$ satisfies $PSU_{3\delta}(q) \leq G/Z(G) \leq \Aut(PSU_{3\delta}(q))$ with $|Z(G)|_3=3$, where $3 \parallel q+1$ and $Z(G) \leq G'$;
\item[(S12)] $G$ satisfies $G_2(q) \leq G \leq \Aut(G_2(q))$ for some $q$ not divisible by $3$, and $B$ is the principal block;
\item[(S13)] $G$ satisfies ${}^2F_4(2^{2m+1}) \leq G \leq \Aut({}^2F_4(2^{2m+1}))$ for some $m \in \mathbb{N}$, and $B$ is the principal block.
\end{enumerate} 
\end{proposition}

\begin{proof}
Consider first $O_p(G)$, which is a subgroup of $D$. We may suppose that $O_p(G) \neq D$ as otherwise we are in case (S1).

Suppose that $|O_p(G)|=p^2$. Now the unique block of $\cO C_G(O_p(G))$ covered by $B$ has central defect group $D \cap C_G(O_p(G)) = O_p(G)$. Hence it is nilpotent and by (R2) $C_G(O_p(G)) = O_p(G)Z(G)$. Hence $B$ has maximal defect. 
Now $O_p(G) \cong C_p \times C_p$, so $G/C_G(O_p(G))$ is isomorphic to a subgroup of $GL_2(p)$ containing a Sylow $p$-subgroup, acting on $O_p(G)$ by the natural representation. Since we do not have $D \lhd G$, by Lemma \ref{subs_PSL2:lemma} this means $SL_2(p) \leq G/C_G(O_p(G)) \leq GL_2(p)$. Since $B$ is reduced we have $O_{p'}(G) \leq Z(G) \cap [G,G]$. The Schur multiplier of $G/O_{p'}(G)$ has trivial $p'$-part, so $O_{p'}(G)=1$ and $C_G(O_p(G))=O_p(G)$. Hence by Lemma \ref{GL_2(p)splitting:lemma} we are in case (S2).

We assume from now on that either $|O_p(G)|=p$, in which case $O_p(G) = Z(D)$, being the unique normal subgroup of $D$ of order $p$, or $O_p(G)=1$. 

Now consider $E(G)$, the layer of $G$, which is the product of the components of $G$. The generalized Fitting subgroup is $F^*(G)=E(G)F(G)$, where $F(G)$ is the Fitting subgroup, the product of the groups $O_{l}(G)$ for primes $l$. See~\cite{asc00} for details. For reduced pairs we have $F(G)=O_p(G)Z(G)$. Note that by~\cite[31.13]{asc00} $C_G(F^*(G)) \leq F^*(G)$. If $E(G)=1$, then $G/Z(G)O_p(G)$ is isomorphic to a subgroup of the $p'$-group $\Aut(C_p)$, so that $O_p(G) \in \Syl_p(G)$, contradicting our assumption that $O_p(G) \neq D$. Hence $E(G) \neq 1$. Write $E(G)=L_1 \circ \cdots \circ L_t$, where the $L_i$ are the components of $G$. Since the $p$-rank of a quotient of a subgroup of $D$ is at most two, it follows from (R2) that $t \leq 2$ (see for example the proof of~\cite[Theorem 3.2]{el18} where a similar argument is given in full detail). 

Suppose that $t=2$, so that $|D \cap E(G)| \geq p^2$. Note that $D$ stabilizes $L_1$ and $L_2$. We cannot have $|D \cap E(G)|=p^2$, since then $D$ would have distinct normal subgroups $L_1 \cap D$ and $L_2 \cap D$ of order $p$, which cannot happen in $p^{1+2}$. Hence $D \leq E(G)$, $O_p(G) \leq Z(E(G))$ and $|L_i \cap D|=p^2$ for each $i$. However this would mean that $D$ is abelian, a contradiction. Hence $t=1$, so $L_1/Z(L_1) \leq G/Z(G)O_p(G) \leq \Aut(L_1/Z(L_1))$.

When $D \leq L_1$, we apply~\cite{ae11} as follows. If $L_1/Z(L_1)$ is sporadic, then it follows from~\cite[Table 1]{ae11} that $G$ is as in (S5)-(S9). By Section 2 of~\cite{ae11} there are no blocks of double covers of alternating groups with defect group $D$. If $L_1$ is of Lie type (in defining or non-defining characteristic), then~\cite[Theorem 4.5]{ae11} (for classical groups), \cite[Theorem 5.1]{ae11} (for exceptional groups) and the discussion following the proof of~\cite[Theorem 5.1]{ae11} (for exceptional covers of simple groups) tells us that we are in one of cases (S3), (S4) or (S10)-(S13). 

Hence we may suppose that $D \cap L_1$ has order dividing $p^2$. Note that we must have $O_p(G) \leq L_1$, otherwise $D$ has distinct normal subgroups $O_p(G)$ and $D \cap L_1$, a contradiction. We must have that $\Out(L_1/Z(L_1))$ has order divisible by $p$, otherwise $[G:L_1]$ would have order prime to $p$, contradicting our assumption that $[D:D \cap L_1]=p$ or $p^2$.

If $L_1$ is a finite group of Lie type defined over a field of characteristic $p$, then $D \cap L_1 \in \Syl_p(L_1)$, so that $L_1$ is $PSL_2(p^2)$ or $SL_2(p^2)$. But then $\Out(PSL_2(p^2))$ has order prime to $p$, a contradiction. If $L_1$ is a sporadic or alternating group, then we have that $\Out(L_1/Z(L_1))$ has order $1$ or $2$, a contradiction.

Suppose that $L_1$ is a finite group of Lie type defined over a field of characteristic different to $p$. We consider the structure of the outer automorphism group, found in~\cite[Table 5]{atlas}. 

From Proposition \ref{Prop:spomega}, if $L_1$ is classical but not special linear, unitary or type $D_4$, then it has outer automorphism group of order prime to $p$, so that $G/L_1$ is a $p'$-group, a contradiction. If $L_1$ is of type $D_4$, then Proposition \ref{Prop:spomega} (c) tells us that $D$ is not a defect group of $G$.

Suppose $L_1$ is special linear or unitary.  Since the symmetry of the Dynkin diagram has order at most $2$ and the diagonal automorphisms contribute a normal cyclic subgroup to the outer automorphism group, it follows that $G$ has a normal subgroup $N$ such that $[N{:}L_1]\mid p$ and $N$ acts on $L_1$ as inner-diagonal automorphisms. By Proposition \ref{Prop:slu} either there is no field automorphism of order $p$ or $D\cap N$ is cyclic. In the former case, $N$ has a $p'$-index in $G$, so that $D \leq N$. But~\cite[Proposition 4.1]{ae11} tells us that (since $D \cap L_1$ is abelian) $D$ must be abelian, a contradiction. In the later case, 
$G/N$ has a cyclic Sylow $p$-subgroup. Since $D\cap N$ is cyclic and normal in $D$,
it follows that $D\cap N=Z(D)$ and $D/(D\cap N)\cong C_p\times C_p$, which is impossible. 

Suppose that $L_1$ is of exceptional type. By Proposition \ref{Prop:excep} unless $L_1$ is as in (b) or (c) of Proposition \ref{Prop:excep}, there are no field automorphisms of order divisible by $p$, and further the outer automorphism group of $L_1$ is prime to $p$, a contradiction. 

Finally suppose that $L_1$ is as in one of the cases within (b) and (c) of Proposition \ref{Prop:excep}. Note that in all these cases $\Out(L_1)$ has a normal Sylow $p$-subgroup, which must necessarily have order $3$ since $B$ is reduced. Let $H$ be the preimage of $O_p(G/L_1)$ in $G$. Then $D \leq H$ and $D$ is abelian by Lemma \ref{Lem:stabfieldautp}, a contradiction.
\end{proof}

{\sc Proof of Theorem}~\ref{main_theorem}
Let $G$ be a finite group and $B$ a block of $\cO G$ with defect group $D \cong p_+^{1+2}$ where $p \geq 5$. By Proposition \ref{reduce:prop} we may assume that $(G,B)$ is reduced. By Proposition \ref{reduced_classification:prop} and Section \ref{background} $B$ is one of the blocks listed.
\hfill $\Box$

\bigskip

{\sc Proof of Theorem}~\ref{p=3reduce:theorem}
By Proposition \ref{reduce:prop}, it suffices to show that there are only finitely many Morita equivalence classes amongst the reduced blocks with defect group $3_+^{1+2}$. By~\cite[Corollary 4.18]{ei21}, it further suffices to show that there are only finitely many Morita equivalence classes amongst the unique blocks of $\cO \langle D^g:g \in G \rangle$ covered by reduced blocks $B$. Examining the list given in Proposition \ref{reduced_classification:prop} this means that we may assume one of the following:

(i) $G=D$;

(ii) $G \cong (C_3 \times C_3) \rtimes SL_2(3)$;

(iii) $G$ is quasisimple.

In cases (i) and (ii) there are clearly only finitely many possible Morita equivalence classes, so it suffices to assume that $G$ is quasisimple.

Following~\cite{eel20}, the $\cO$-Morita Frobenius number ${\rm mf}_\cO(B)$ of a block $B$ is defined to be the smallest integer $n$ such that $B^{(p^n)}$ is Morita equivalent to $B$, where $B^{(p^n)}$ is defined as follows. Define a ring automorphism $\bar{\sigma}:kG \rightarrow kG$ by $\bar{\sigma}(\sum_{g \in G} a_g g) = \sum_{g \in G} (a_g)^p g$. Write $B^{(p^n)}$ for the lift of the image of $k \otimes B$ under $\bar{\sigma}^n$. By~\cite[Corollary 3.11]{eel20}, in order to show that there are only finitely many Morita equivalence classes amongst blocks of quasisimple groups with defect group $3_+^{1+2}$, it suffices to show that there integers $c$ and $m$ such that for all blocks $B$ of quasisimple groups with defect group $3_+^{1+2}$ we have ${\rm c}(B) \leq c$ and ${\rm mf}_\cO (B) \leq m$. By~\cite{fk19} we may take $m=4$, and we are done.


\section{Blocks with defect group $5_+^{1+2}$}
\label{p=5}

We apply Theorem \ref{main_theorem} in the case $p=5$, and classify the Morita equivalence classes. Note that it may in theory be possible that these Morita equivalence classes contain blocks with defect group $5_-^{1+2}$, although we do not know of such an example.

Referring to Section \ref{background}, the possible inertial quotients correspondent to conjugacy classes of subgroups of $GL_2(5)$ of order prime $5$, of which there are thirty, obtained by a routine calculation in MAGMA~\cite{magma}. Each block with normal defect group is Morita equivalent to a block of $5_+^{1+2} \rtimes \hat{E}$ where $E$ is the inertial quotient and $\hat{E}$ is a central extension of $E$ by $Z \leq Z(\hat{E}) \cap \hat{E}'$ acting trivially on $5_+^{1+2}$. In order to describe all blocks with normal defect group, for each $E$ we choose a central extension with $Z \cong H^2(E,k^\times)$, which we obtain via the pMultiplicator function in MAGMA. To differentiate between different actions of isomorphic groups $E$ on $D=5_+^{1+2}$ we use the notation $(E)_i$. In our case all $D \rtimes E$ where there is such an ambiguity have order 2000 or less, and so we specify the action by identifying $D \rtimes E$ in the SmallGroup library of MAGMA (or equivalently GAP~\cite{gap_SG}). Table \ref{Order125_normal} gives each of the thirty possibilities, along with: SmallGroup representation of $5_+^{1+2} \rtimes E$ (when this is at most 2000); the isomorphism type of $H^2(E,k^\times)$; a candidate for $\hat{E}$ given by its SmallGroup labelling; and the number of Morita equivalence classes amongst the blocks of $5_+^{1+2} \rtimes \hat{E}$. Note that when $|H^2(E,k^\times)|=2$ it is easy to verify that the two blocks of $5_+^{1+2} \rtimes \hat{E}$ are in distinct Morita equivalence classes by checking, using the given SmallGroup representation of $\hat{E}$, that the number of simple modules in each block is different. There are two cases, treated in Lemma \ref{C4xC4} and \ref{C4circD8}, that have larger Schur multipliers and require us to establish Morita equivalences between nonprincipal blocks.

In Table \ref{Order125_blocks_normal} we give the numbers of irreducible characters of height zero $k_0(B)$, height one $k_1(B)$, and the number of irreducible Brauer characters $l(B)$ for each Morita equivalence class of blocks with normal defect group. Recall that the height of an irreducible character $\chi$ in a block with defect group $D$ is the non-negative integer $h$ such that $\chi(1)_p=p^h[G:D]_p$. 

We now make some comments on the calculation of $k_0(B)$, $k_1(B)$ and $l(B)$ in Table \ref{Order125_blocks_normal}. Where $|D \rtimes \hat{E}| \leq 2000$ we used a SmallGroup representation of the group and computed character tables using GAP. Otherwise, using GAP we have formed the holomorph of $5_+^{1+2}$ and found $D \rtimes E$ through a search of its subgroups. We computed the inertia subgroups of the nontrivial irreducible characters of $D$ in $E$, which in all cases were groups with trivial Schur multiplier. From this we could deduce the number of irreducible characters in each block.  

\begin{table}
\centering
\begin{tabular}{|l|l|l|l|l|l|}
\hline
$E$ & $D \rtimes E$ as & $H^2(E,k^\times)$ & $\hat{E}$ as & $\#$Morita eq. & Notes \\
 & SmallGroup & & SmallGroup & classes & \\
\hline
$1$ & (125,3) & 1 & (1,1) & 1 & \\
$(C_2)_1$ & (250,5) & 1 & (10,2) & 1 & \\
$(C_2)_2$ & (250,8) & 1 & (2,1) & 1 & \\
$C_3$ & (375,2) & 1 & (3,1) & 1 & \\
$(C_4)_1$ & (500,17) & 1 & (4,1) & 1 & \\
$(C_4)_2$ & (500,21) & 1 & (4,1) & 1 & \\
$(C_4)_3$ & (500,23) & 1 & (4,1) & 1 & \\
$(C_4)_4$ & (500,25) & 1 & (4,1) & 1 & \\
$C_2 \times C_2$ & (500,27) & $C_2$ & (8,4) & 2 & \\
$C_6$ & (750,6) & 1 & (6,2) & 1 & \\
$S_3$ & (750,5) & 1 & (6,1) & 1 & \\
$C_8$ & (1000,86) & 1 & (8,1) & 1 & \\
$(C_4 \times C_2)_1$ & (1000,89) & $C_2$ & (16,6) & 2 & \\
$(C_4 \times C_2)_2$ & (1000,91) & $C_2$ & (16,6) & 2 & \\
$D_8$ & (1000,92) & $C_2$ & (16,7) & 2 & \\
$Q_8$ & (1000,93) & 1 & (8,4) & 1 & \\
$C_3 \rtimes C_4$ & (1500,35) & 1 & (12,1) & 1 & \\
$C_{12}$ & (1500,36) & 1 & (12,2) & 1 & \\
$D_{12}$ & (1500,37) & $C_2$ & (24,6) & 2 & \\
$C_4 \times C_4$ & (2000,473) & $C_4$ & (64,18) & 3 & Lemma \ref{C4xC4} \\
$M_4(2)$ & (2000,474) & 1 & (16,6) & 1 & \\
$C_4 \circ D_8$ & (2000,475) & $C_2 \times C_2$ & (64,74) & 2 & Lemma \ref{C4circD8} \\
$C_3 \rtimes C_8$ & N/A & 1 & (24,1) & 1 & \\
$C_{24}$ & N/A & 1 & (24,2) & 1 & \\
$SL_2(3)$ & N/A & 1 & (24,3) & 1 & \\
$C_4 \times S_3$ & N/A & $C_2$ & (48,5) & 2 & \\
$C_4 \wr C_2$ & N/A & $C_2$ & (64,10) & 2 & \\
$C_8 \rtimes S_3$ & N/A & 1 & (48,5) & 1 & \\
$C_4 \circ SL_2(3)$ & N/A & 1 & (48,33) & 1 & \\
$U_2(3)$ & N/A & 1 & (96,67) & 1 & \\
\hline
\end{tabular}
\caption{Inertial quotients for blocks with defect group $D \cong 5_+^{1+2}$}
\label{Order125_normal}
\end{table}

\begin{table}
\centering
\begin{tabular}{|l|l|l|l|}
\hline
$E$ & $k_0(B)$ & $k_1(B)$ & $l(B)$ \\
\hline
$1$ & 20 & 4 & 1 \\
$(C_2)_1$ & 20 & 2 & 2 \\
$(C_2)_2$ & 14 & 8 & 2 \\
$C_3$ & 11 & 12 & 3 \\
$(C_4)_1$ & 25 & 1 & 4 \\
$(C_4)_2$ & 13 & 1 & 4 \\
$(C_4)_3$ & 10 & 4 & 4 \\
$(C_4)_4$ & 10 & 16 & 4 \\
$C_2 \times C_2$ & 16,13 & 4,4 & 4,1 \\
$C_6$ & 10 & 24 & 6 \\
$S_3$ & 13 & 6 & 3 \\
$C_8$ & 11 & 2 & 8 \\
$(C_4 \times C_2)_1$ & 20,14 & 2,2 & 8,2 \\
$(C_4 \times C_2)_2$ & 14,8 & 8,8 & 8,2 \\
$D_8$ & 14,11 & 8,8 & 5,2 \\
$Q_8$ & 8 & 20 & 5 \\
$C_3 \rtimes C_4$ & 8 & 24 & 6 \\
$C_{12}$ & 14 & 12 & 12 \\
$D_{12}$ & 14,11 & 12,12 & 6,3 \\
$C_4 \times C_4$ & 25,10,13 & 4,4,4 & 16,1,4 \\
$M_4(2)$ & 13 & 4 & 10 \\
$C_4 \circ D_8$ & 16,10 & 10,10 & 10,4 \\
$C_3 \rtimes C_8$ & 13 & 6 & 12 \\
$C_{24}$ & 25 & 6 & 24 \\
$SL_2(3)$ & 8 & 28 & 7 \\
$C_4 \times S_3$ & 16,10 & 12,12 & 12,6 \\
$C_4 \wr C_2$ & 20,11 & 5,5 & 14,5 \\
$C_8 \rtimes S_3$ & 20 & 6 & 18 \\
$C_4 \circ SL_2(3)$ & 16 & 14 & 14 \\
$U_2(3)$ & 20 & 7 & 16 \\

\hline
\end{tabular}
\caption{Morita equivalence classes of blocks with defect group $5_+^{1+2}$: blocks of groups with a normal defect group}
\label{Order125_blocks_normal}
\end{table}

\begin{lemma}
\label{C4xC4}
There are precisely three Morita equivalence classes of blocks with normal defect group $D \cong 5_+^{1+2}$ and inertial quotient $C_4 \times C_4$.
\end{lemma}

\begin{proof}
We may choose a central extension $\hat{E}$ of $C_4 \times C_4$ by $Z \cong C_4$ identified as SmallGroup(64,18). Write $\Irr(Z) = \{1_Z=\lambda^0,\lambda, \lambda^2,\lambda^3\}$, so that $\lambda^2$ is non-faithful, and write $B_{i}$ for the block of $D \rtimes \hat{E}$ covering $\lambda^i \in \Irr(Z)$. Calculations using GAP give us that $l(B_0)= 16$, $l(B_1)=l(B_{3})=1$ and $l(B_2)=4$, so that $B_0$, $B_1$ and $B_2$ lie in different Morita equivalence classes. Now we may embed $D \rtimes \hat{E}$ in $D \rtimes H$, where $H \cong (C_2)^3.D_{16}$, identified as SmallGroup(128,71), in which $\lambda$ and $\lambda^3$ are permuted. Hence $B_1$ and $B_3$ are Morita equivalent.
\end{proof} 

\begin{lemma}
\label{C4circD8}
There are precisely two Morita equivalence classes of blocks with normal defect group $D \cong 5_+^{1+2}$ and inertial quotient $C_4 \circ D_8$.
\end{lemma}

\begin{proof}
We may choose a central extension $\hat{E}$ of $C_4 \circ D_8$ by $Z \cong C_2 \times C_2$ identified as SmallGroup(64,73) and embed it as a normal subgroup in a group $H$ identified as SmallGroup(192,194), in which $Z$ is a noncentral but normal subgroup. Note that $H/Z \cong C_4.A_4$, which is SmallGroup(48,33) and is a subgroup of $GL_2(5)$, so $H$ acts on $D$ with kernel $Z$. Hence we may embed $D \rtimes \hat{E} \leq D \rtimes H$. The blocks of $D \rtimes \hat{E}$ correspond to the irreducible characters of $Z$. Calculations using GAP show that $H$ permutes the non-trivial irreducible characters of $Z$, and so the three non-principal blocks of $D \rtimes \hat{E}$ are Morita equivalent.
\end{proof}

Regarding blocks with non-normal defect group, using~\cite{gap} and~\cite{gap_char} we may obtain decomposition matrices for all blocks identified in Theorem \ref{main_theorem}(b)-(e) aside from $Th$. One may check that all these decomposition matrices are (often subtly) distinct, and further cannot be the decomposition matrix for a block with normal defect group since there is always an irreducible Brauer character that does not lift to an ordinary irreducible character. In the case of $Th$, where no decomposition matrix is currently available, we may use its character table and the fact that all nonprincipal $5$-blocks have defect zero to deduce $k(B_0(Th))$ and $l(B_0(Th))$ (see~\cite{atlas} and~\cite{un04}). We see that no other block with this defect group has the same size decomposition matrix. Table \ref{Order125classify_non-normal} lists the Morita equivalence classes we obtain from blocks with non-normal defect groups. A subtlety in the definition of the groups in cases (c) and (d) of Theorem \ref{main_theorem} is that we are taking a central extension of an automorphism group of a simple group, and so the Schur multiplier may in general be different to that of the simple group itself. However, $\Out(PSL_3(5))$ and $\Out(PSU_3(5))$ are both cyclic, so that by, for example,~\cite[Lemma 3.4(ii)]{ea01} the Schur multiplier of $X$ for $PSL_3(5) \leq X \leq \Aut(PSL_3(5))$ or $PSU_3(5) \leq X \leq \Aut(PSU_3(5))$ is a homomorphic image of that for $PSL_3(5)$ or $PSU_3(5)$ respectively.

\begin{table}
\centering
\begin{tabular}{|l|l|l|l|l|l|}
\hline
Representative $B$ & $k_0(B)$ & $k_1(B)$ & $l(B)$ & Inertial quotient & Notes\\
\hline
$(C_5)^2 \rtimes SL_2(5)$ & 13 & 1 & 5 & $(C_4)_2$ & \\
$(C_5)^2 \rtimes SL_2(5).2$ & 20 & 2 & 10 & $(C_4 \times C_2)_1$ & \\\
$(C_5)^2 \rtimes GL_2(5)$ & 25 & 4 & 20 & $C_4 \times C_4$ & \\
$B_0(SL_3(5))$ & 25 & 4 & 24 & $C_4 \times C_4$ & \\
$B_0(SL_3(5).2)$ & 20 & 5 & 18 & $C_4 \wr C_2$ & \\
$B_0(SU_3(5))$ & 11 & 2 & 8 & $C_8$ & \\
$B_1(SU_3(5))$ & 11 & 2 & 8 & $C_8$ & Remark \ref{non-normal:rem}(i) \\
$B_0(SU_3(5).2)$ & 13 & 4 & 10 & $M_4(2)$ & \\
$B_0(SU_3(5).3)$ & 25 & 6 & 24 & $C_{24}$ & \\
$B_1(SU_3(5).3)$ & 25 & 6 & 24 & $C_{24}$ & Remark \ref{non-normal:rem}(ii) \\
$B_0(HS)$ & 13 & 4 & 10 & $M_4(2)$ & \\
$B_1(2.HS)$ & 13 & 4 & 10 & $M_4(2)$ & \\
$B_0(HS.2)$ & 20 & 5 & 14 & $C_4 \wr C_2$ & \\
$B_1(2.HS.2)$ & 11 & 5 & 5 & $C_4 \wr C_2$ & \\
$B_0(McL)$ & 13 & 6 & 12 & $C_3 \rtimes C_8$ & \\
$B_1(3.McL)$ & 13 & 6 & 12 & $C_3 \rtimes C_8$ & Remark \ref{non-normal:rem}(iii) \\
$B_0(McL.2)$ & 20 & 6 & 18 & $C_8 \rtimes S_3$ & \\
$B_0(Ru)$ & 20 & 5 & 18 & $C_4 \wr C_2$ & \\
$B_1(2.Ru)$ & 20 & 5 & 18 & $C_4 \wr C_2$ & \\
$B_0(Co_3)$ & 20 & 6 & 18 & $C_8 \rtimes S_3$ & \\
$B_0(Co_2)$ & 20 & 7 & 16 & $U_2(3)$ & \\
$B_0(Th)$ & 20 & 7 & 20 & $U_2(3)$ & \\
\hline
\end{tabular}
\caption{Morita equivalence classes of blocks with defect group $P=5_+^{1+2}$: blocks of groups with non-normal defect groups}
\label{Order125classify_non-normal}
\end{table}

\begin{remarks}
\label{non-normal:rem}
(i) The two nonprincipal blocks of maximal defect of $3.SU_3(5)$ are conjugate, and so Morita equivalent, in $3.SU_3(5).2$.

(ii) The two nonprincipal blocks of maximal defect of $3.SU_3(5).3$ are conjugate, and so Morita equivalent, in $3.SU_3(5).6$.

(iii) The two nonprincipal blocks of maximal defect of $3.McL$ are conjugate, and so Morita equivalent, in $3.McL.2$.

\end{remarks}


\section{The Alperin-McKay conjecture}
\label{counting_conjectures}

In this section we prove Corollary \ref{alp_mckay}. In Table \ref{pgeq7numerical} we give data on the blocks in Theorem \ref{main_theorem}(f-h). Inertial quotients were computed using~\cite{atlas},~\cite{wi98} and GAP, searching within the holomorph of $p_+^{1+2}$ for appropriate subgroups where data in~\cite{atlas} and~\cite{wi98} was insufficient.

\begin{table}
\centering
\begin{tabular}{|l|l|l|l|l|l|}
\hline
$p$ & $B$ & $k_0(B)$ & $k_1(B)$ & Inertial quotient & Notes\\
\hline
7 & $B_0(He)$ & 20 & 3 & $S_3 \times C_3$ & \\
7 & $B_0(He.2)$ & 25 & 6 & $S_3 \times C_6$ & \\
7 & $B_0(O'N)$ & 20 & 4 & $C_3 \times D_8$ & \\
7 & $B_0(3.O'N)$ & 20 & 4 & $C_3 \times D_8$ & Remark \ref{non-normal:remII}(i) \\
7 & $B_0(O'N.2)$ & 25 & 8 & $C_3 \times D_{16}$ & \\
7 & $B_0(Fi_{24}')$ & 25 & 6 & $S_3 \times C_6$ & \\
7 & $B_1(3.Fi_{24}')$ & 25 & 6 & $S_3 \times C_6$ & Remark \ref{non-normal:remII}(ii) \\
7 & $B_0(Fi_{24}'.2)$ & 35 & 6 & $C_3 \times (C_3 \rtimes D_8)$ & \\
11 & $B_0(J_4)$ & 42 & 7 & $C_5 \times GL_2(3)$ & \\
13 & $B_0(M)$ & 55 & 7 & $C_3 \times U_2(3)$ & \\
\hline
\end{tabular}
\caption{Blocks of sporadic almost simple groups with defect group $p_+^{1+2}$ for $p \geq 7$}
\label{pgeq7numerical}
\end{table}

\begin{remarks}
\label{non-normal:remII}
(i) The two nonprincipal blocks of maximal defect of $3.O'N$ are conjugate, and so Morita equivalent, in $3.O'N.2$.

(ii) The two nonprincipal blocks of maximal defect of $3.Fi_{24}'$ are conjugate, and so Morita equivalent, in $3.Fi_{24}'.2$.

\end{remarks}




Next we compute invariants in case (b) of Theorem \ref{main_theorem}.

\begin{lemma}
\label{p^2SL_2(p)C}
Consider $G=SL_2(p)C=SL_2(p) \rtimes C$ where $C \leq C_{p-1}$ so $SL_2(p) \leq G \leq GL_2(p)$. Let $B$ be the unique $p$-block of $G$. Let $D \in \Syl_p(G)$ and let $b$ be the unique block of $N_G(D)$.

Write $c=|C|$. Let $\psi:GL_2(p) \rightarrow \Out(SL_2(p))$, defined via conjugation, so that $\psi(C)$ has order dividing $2$.

\begin{enumerate}
\item[$(a)$]
(i) If $|\psi(C)|=1$, then $\Irr(B)$ consists of: $c$ linear characters; $2c$ of degree $(p+1)/2$; $2c$ of degree $(p-1)/2$; $(p-3)c/2$ of degree $p+1$; $(p-1)c/2$ of degree $(p-1)$; $c$ of degree $p$. Hence $k_0(B)=(p+4)c+ (p-1)/c$ and $k_1(B)=c$. 

(ii) If $|\psi(C)|=2$, then $\Irr(B)$ consists of: $c$ linear characters; $(p-2)c/2$ of degree $p+1$; $pc/2$ of degree $p-1$; $c$ of degree $p$. Hence $k_0(B)=(p+1)c+ (p-1)/c$ and $k_1(B)=c$.

In either case $l(B)=(p-1)c$.

\item[$(b)$] (i) If $|\psi(C)|=1$, then $\Irr(b)$ consists of: $(p-1)c$ linear characters; $4c$ of degree $(p-1)/2$; $c$ of degree $p-1$; $(p-1)/c$ of degree $(p-1)c$; $c$ of degree $p(p-1)$. Hence $k_0(b)=(p+4)c+ (p-1)/c$ and $k_1(B)=c$. 

(ii) If $|\psi(C)|=2$, then $\Irr(b)$ consists of: $(p-1)c$ linear characters; $2c$ of degree $p-1$; $(p-1)/c$ of degree $(p-1)c$; $c$ of degree $p(p-1)$. Hence $k_0(b)=(p+1)c+ (p-1)/c$ and $k_1(B)=c$.

In either case $l(b)=(p-1)c$.
\end{enumerate}
\end{lemma}

{\sc Proof of Corollary \ref{alp_mckay}.}
By Theorem \ref{main_theorem} $B$ is Morita equivalent to one of the blocks listed with the same inertial quotient and K\"ulshammer-Puig class, and the same is true of the Brauer correspondent block $b$ of $N_G(D)$. Therefore it suffices to check the conjecture for the listed blocks. For those of type (a) the result is immediate, and for type (b) it follows from Lemma \ref{p^2SL_2(p)C}. For types (c) and (d), the result follows from~\cite{mas10}, noting that the equivariant bijections constructed in these cases are indeed between $B$ and $b$. Since $\Out(PSL_3(p))$ and $\Out(PSU_3(p))$ both have cyclic Sylow $l$-subgroups for all prime $l$ (so that irreducible characters of the quasisimple group extend to their stabilizers), the equivariant bijection implies immediately that the conjecture holds in this case. For type (e) the conjecture follows by comparing Tables \ref{Order125classify_non-normal} and \ref{Order125_blocks_normal}. For the remaining cases we refer to Table \ref{pgeq7numerical}. To determine $k_0(b)$ and verify that it is equal to $k_0(B)$, as in Section \ref{p=5} we used GAP to form the holomorph of $p_+^{1+2}$ and found $D \rtimes E$ through a search of its subgroups. We computed the inertia subgroups of the nontrivial irreducible characters of $D$ in $E$, which in all cases were groups with trivial Schur multiplier. From this we could deduce $k_0(b)$. \hfill $\Box$






\begin{center} {\bf Acknowledgments}
\end{center}
\smallskip

We thank Shigeo Koshitani for some useful discussions and for drawing our attention to~\cite{nu09}.


\end{document}